\documentclass{article}

\usepackage{paralist}
\title{Space-time finite element discretization of\\ parabolic optimal control problems\\ with energy regularization}
\author{Ulrich~Langer\footnote{Johann Radon Institute for Computational and
Applied Mathematics, Austrian Academy of Sciences, Altenberger Stra{\ss}e 69, 
4040 Linz, Austria, Email: ulrich.langer@ricam.oeaw.ac.at}, \; 
Olaf~Steinbach\footnote{Institut f\"{u}r Angewandte Mathematik,
Technische Universit\"{a}t Graz, Steyrergasse 30, 8010 Graz, Austria,
Email: o.steinbach@tugraz.at}, \; 
Fredi~Tr\"{o}ltzsch\footnote{Institut f\"{u}r Mathematik, Technische
Universit\"{a}t Berlin, Stra{\ss}e des 17. Juni 136, 10623 Berlin,
Germany, Email: troeltzsch@math.tu-berlin.de}, \; 
Huidong~Yang\footnote{Johann Radon Institute for Computational and Applied
Mathematics, Austrian Academy of Sciences, Altenberger Stra{\ss}e 69, 
4040 Linz, Austria, Email: huidong.yang@ricam.oeaw.ac.at} 
}

\date{\today}

\usepackage{amsmath}
\usepackage{amsthm}
\usepackage{amsfonts}
\usepackage{booktabs}
\usepackage{graphicx}

\usepackage{multirow}
\usepackage{hyperref}

\usepackage[ulem=normalem,draft]{changes}

\newtheorem{theorem}{Theorem}
\newtheorem{lemma}{Lemma}

\newcommand{\fredi}[1]{{\color{cyan}{#1}}}

\newcommand{\ignore}[1]{}

\begin{document}

\maketitle

\begin{abstract}
We analyze space-time finite element methods for the 
numerical solution of distributed parabolic optimal control problems 
with energy regularization in the Bochner space $L^2(0,T;H^{-1}(\Omega))$. 
By duality, the related norm can be evaluated by means of the solution 
of an elliptic quasi-stationary boundary value problem. When eliminating 
the control, we end up with the reduced optimality system that is nothing 
but the variational formulation of the coupled forward-backward primal 
and adjoint equations. Using Babu\v{s}ka's theorem, we prove unique 
solvability in the continuous case. Furthermore, we establish the discrete 
inf-sup condition  for  any conforming space-time finite element
discretization yielding quasi-optimal discretization error estimates.
Various numerical examples confirm the theoretical findings. We emphasize 
that the energy regularization results in a more localized control with 
sharper contours for discontinuous target functions, which is demonstrated 
by a comparison with an $L^2$ regularization and with a sparse optimal 
control approach.
\end{abstract} 


\begin{msc}
35K20, 49J20, 65M15, 65M50, 65M60
\end{msc}

\section{Introduction}\label{sec:intro}
In this paper, we consider and analyze continuous space-time finite element 
methods on fully unstructured simplicial space-time meshes for the numerical 
solution of the following parabolic optimal control problem: For a  given 
target function $u_d \in L^2(Q)$, we want to minimize the cost functional
\begin{equation}\label{eqn:energy regularization}
{\mathcal J}(u,z):=\frac{1}{2}\int_Q
\left|u-{u}_d\right|^2 \, dx \, dt + \frac{1}{2}\varrho \, 
\|z\|_{L^2(0,T;H^{-1}(\Omega))}^2 
\end{equation}
subject to the linear parabolic state equation
\begin{equation}
\label{eqn:linearstateequation}
\partial_t u-\Delta_x u  =  z  \textup{ in } Q,\quad
u  =  0 \textup{ on } \Sigma,\quad
u = 0  \textup{ on } \Sigma_0 ,
\end{equation}
where $Q := \Omega \times (0,T)$ is the space-time 
domain 
with the
lateral boundary $\Sigma := \partial \Omega \times (0,T)$, and
$\Sigma_0 := \Omega \times \{ 0 \}$. 
Moreover, $\Omega \subset {\mathbb{R}}^d$, $d=2,3$, is a bounded Lipschitz domain,
$T>0$ is the final time, and $\varrho > 0$ is some regularization parameter.

In our recent paper \cite{LSTY_SISC_arXiv:2020}, we have considered the related
standard optimal control problem with regularization in $L^2(Q)$, i.e.,
\begin{equation}\label{eqn:L2 regularization}
{\mathcal J}(u,z):=\frac{1}{2}\int_Q
\left|u-{u}_d\right|^2 \, dx \, dt + \frac{1}{2}\varrho \, 
\|z\|_{L^2(Q)}^2 \, ,
\end{equation}
subject to \eqref{eqn:linearstateequation}. 
In both cases, we can numerically solve the corresponding parabolic 
forward-backward optimality systems at once. This allows not only for a more 
efficient solution of the global system, but also for parallelization in space 
and time, and for adaptive discretizations 
simultaneously in space and time. In contrast to classical time-stepping 
methods or discontinuous Galerkin (dG) methods which are defined with respect 
to time slices or slabs, see, e.g., 
the monographs \cite{LSTY:Lang:2000a} and \cite{LSTY:Thomee2006a},
and the review article \cite{LSTY:Gander:2015a} on parallel-in-time methods,
we use fully unstructured simplicial space-time meshes for the numerical
solution of the parabolic state equation (\ref{eqn:linearstateequation}),
see the recent review  article \cite{OSHY19} and the related 
references therein. 

The standard approach for distributed control problems is to consider the
control $z$ in $L^2(Q)$. There is a huge number of publications on the standard 
setting \eqref{eqn:L2 regularization} with $L^2(Q)$-regularization. We here
only refer to the monographs
\cite{LSTY:BorziSchulz:2011a,LSTY:HinzePinnauUlbrichUlbrich:2009a,FT10},
to the more recent papers \cite{WGMHZZ2012,DMBV2008a,DMBV2008b} on 
discontinuous (dG) and continuous Galerkin time-slice finite element methods,
\cite{LSTY:Neumueller:2010} on full space-time dG finite element methods, 
\cite{LSTY:GunzburgerKunoth:2011a} on space-time adaptive wavelet methods,
\cite{LSTY:KollmannKolmbauerLangerWolfmayrZulehner:2013a,
LSTY:LangerRepinWolfmayr:2016a} on multiharmonic methods,
\cite{AASV2015} on proper orthogonal decomposistion,
\cite{LSTY:BuengerDolgovStoll:2020a} on low-rank tensor method, and
to our very recent paper \cite{LSTY_SISC_arXiv:2020} on completely  
unstructured space-time finite element methods for optimal
control of parabolic equations based on $L^2(Q)$-regularization,
and the references given therein.
However, since the state $u \in L^2(0,T;H^1_0(\Omega))$ is well defined 
as the solution of the forward heat equation for $z \in L^2(0,T;H^{-1}(\Omega))$,
we may also consider the tracking type functional as given in 
\eqref{eqn:energy regularization}. Applying integration by parts also in time
to derive a variational formulation for the adjoint equation, we end up,
in contrast to the case of $L^2$ regularization,
with a positive definite but skew-symmetric bilinear form describing the
optimality system. In this paper, we provide a complete numerical analysis
for both the continuous and discrete system. 

The rest of the  paper is structured as follows. 
In Section~\ref{sec:prelim}, we introduce some notation and state some 
preliminary results on the solvability and numerical analysis of the 
parabolic initial-boundary value problem that serves as state equation 
in the optimal control problem. In Section~\ref{sec:sthm1}, we analyze
the unique solvability of the continuous optimality system, whereas 
Section~\ref{Section:Discretization} is devoted to the numerical analysis
of the space-time finite element approximation. Numerical results are
presented in Section~\ref{Section:Numerical Examples}
Finally, some conclusions are drawn in Section \ref{sec:con}. 

\section{Preliminaries}\label{sec:prelim}
In this section, we introduce basic notations, and summarize some recent 
results on space-time finite element methods for the numerical solution of 
the state equation \eqref{eqn:linearstateequation}. For the mathematical
analysis of parabolic initial boundary value problems in  
space-time Sobolev spaces, see \cite{LSTY:Ladyzhenskaya:1985a,
LSTY:LadyzhenskayaetSolonnikovUraltseva:1967a}, and
\cite{LSTY:Lions:1968a,LSTY:Zeidler:1990a} for 
Bochner spaces of abstract functions, mapping the time 
interval $(0,T)$ to some Hilbert or Banach space.

Following the latter approach, we define
\begin{eqnarray*}
X & := & L^2(0,T; H_0^1(\Omega)) \cap H_{0,}^1(0,T;H^{-1}(\Omega)) \\
  & = & \Big \{ v \in L^2(0,T; H_0^1(\Omega)) : \partial_t v \in 
        L^2(0,T;H^{-1}(\Omega)), \, v = 0 \mbox{ on }  \Sigma_0 \Big \},\\
Y & := & L^2(0,T;H_0^1(\Omega)), \qquad
Y^*:= L^2(0,T;H^{-1}(\Omega)),
\end{eqnarray*}
using the standard Sobolev spaces $H^1_0(\Omega)$ and its dual
$H^{-1}(\Omega)$. Note that we have 
$X = \{ v \in W(0,T):\, v = 0 \mbox{ on }  \Sigma_0\}$ 
as used in \cite{LSTY:Lions:1968a}. The related norms are given by
\[
\| u \|_X := \Big[ \| w_u \|^2_Y + \| u \|^2_Y \Big]^{1/2}, \quad
\| v \|_Y := \| \nabla_x v \|_{L^2(Q)}, 
\]
where $w_u \in Y$ is the unique solution of the variational 
formulation \cite{OS15}
\begin{equation}\label{eqn:Definition w H-1}
\int_Q \nabla_x w_u \cdot \nabla_x v \, dx \, dt = 
\langle \partial_t u , v \rangle_Q , \quad \forall v
\in Y.
\end{equation}
The standard weak formulation of the initial boundary value 
problem~\eqref{eqn:linearstateequation} reads as follows:
Given $z \in Y^*$, find $u \in X$ such that
\begin{equation}\label{eqn:stateequation} 
b(u,v) = \langle z, v \rangle_Q, \quad \forall v \in Y, 
\end{equation}
with the bilinear form $b(\cdot,\cdot): X \times Y \to \mathbb{R}$,
\begin{equation}
\label{eqn:bilinearform_b(.,.)} 
b(u,v) := \int_Q \Big[
\partial_t u \, v + \nabla_x u \cdot \nabla_x v \Big] \, dx \, dt, \quad 
\forall (u,v) \in X \times Y, 
\end{equation}
and the linear form $\langle z, \cdot \rangle_Q : Y \to \mathbb{R}$ 
with the duality pairing $\langle z,v \rangle_Q$
as extension of the inner product in $L^2(Q)$. Similarly, the first 
integral in \eqref{eqn:bilinearform_b(.,.)} has to be understood as 
duality pairing as well.

The bilinear form $b(\cdot,\cdot)$ is bounded,
\begin{equation}\label{eqn:boundedness_b(.,.)}
|b(u,v)| \leq \sqrt{2} \, \| u \|_X \| v \|_Y , \quad 
\forall (u,v) \in X \times Y,
\end{equation}
and satisfies the inf-sup stability condition \cite[Theorem 2.1]{OS15}
\begin{equation}\label{eqn:infsup_b(.,.)}
\inf\limits_{0 \neq u \in X}\sup\limits_{0 \neq v \in Y}
\frac{b(u,v)}{\| u \|_X \| v \|_Y} \ge \frac{1}{2\sqrt{2}} . 
\end{equation}
Moreover, for $v \in Y \setminus \{0\}$, we define
\[
\widetilde{u}(x,t) = \int_0^t v(x,s) \, ds, \quad (x,t) \in Q 
\]
to obtain
\[
b(\widetilde{u},v) = \| v \|^2_{L^2(Q)} +
\frac{1}{2} \, \| \nabla_x \widetilde{u}(T) \|^2_{L^2(\Omega)} > 0.
\]
Hence, we can apply the Ne\u{c}as-Babu\v{s}ka theorem 
\cite{IB71,LSTY:Necas:1962a} to conclude that the variational problem 
\eqref{eqn:stateequation} is well-posed, see also
\cite{LSTY:BabuskaAziz:1972a,LSTY:Braess:2007a,LSTY:ErnGuermond:2004a}.

For the finite element discretization of the variational formulation 
(\ref{eqn:stateequation}), we introduce conforming space-time finite 
element spaces $X_h \subset X$ and $Y_h \subset Y$, 
where we assume $X_h \subseteq Y_h$. In particular, we may use 
$X_h = Y_h = S_h^1(Q_h) \cap X$ spanned by continuous and piecewise 
linear basis functions which are defined with respect to some 
admissible decomposition $\mathcal{T}_h(Q)$ of the space-time domain 
$Q$ into shape regular simplicial finite elements $\tau_\ell$, and which 
are zero at the initial time $t=0$ and at the lateral boundary $\Sigma$,
where  $h$ denotes a suitable mesh-size parameter, see, e.g., 
\cite{LSTY:Braess:2007a,LSTY:ErnGuermond:2004a,OS15}.
Then the finite element approximation of 
(\ref{eqn:stateequation}) is to find $u_h \in X_h$ such that
\begin{equation}\label{eqn:stateequation FEM}
b(u_h,v_h) = \langle z , v_h \rangle_Q, \quad \forall v_h \in Y_h .
\end{equation}
When replacing (\ref{eqn:Definition w H-1})
by its finite element approximation to find $w_{u,h} \in Y_h$ such that
\begin{equation}\label{eqn:Definition w H-1 FEM}
\int_Q \nabla_x w_{u,h} \cdot \nabla_x v_h \, dx \, dt =
\int_Q \partial_t u \, v_h \, dx \, dt, \quad \forall v_h \in Y_h,
\end{equation}
we can define a discrete norm
\[
\| u \|_{X_h} \, := \Big[ \| w_{u,h} \|_Y^2 + \| u \|_Y^2 \Big]^{1/2} \, .
\]
As in the continuous case, see (\ref{eqn:infsup_b(.,.)}),
we can prove a discrete inf-sup condition, see \cite[Theorem 3.1]{OS15},
\begin{equation}\label{eqn:bilinear form heat discrete stability}
\frac{1}{2\sqrt{2}} \, \| u_h \|_{X_h} \leq 
\sup\limits_{0 \neq v_h \in Y_h} \frac{b(u_h,v_h)}{\| v_h \|_Y} , \quad
\forall u_h \in X_h .
\end{equation}
Hence, we conclude unique solvability of the Galerkin scheme 
(\ref{eqn:stateequation FEM}), and we obtain the following
quasi-optimal error estimate, see \cite[Theorem 3.2]{OS15}:
\begin{equation}\label{eqn:stateequation FEM Cea}
\| u - u_h \|_{X_{0,h}} \leq 5 \inf\limits_{z_h \in X_{0,h}} \| u - z_h \|_{X_0} \, .
\end{equation}
In particular, when assuming $u \in H^2(Q)$, this finally results in the
energy error estimate, see \cite[Theorem 3.3]{OS15},
\begin{equation}
\| u - u_h \|_{L^2(0,T;H^1_0(\Omega))} \leq c \, h \, |u|_{H^2(Q)} \, .
\end{equation}

\section{The first-order optimality system}
\label{sec:sthm1}
We now consider the optimal control problem to minimize 
\eqref{eqn:energy regularization} subject to the heat equation
\eqref{eqn:linearstateequation}. As in \eqref{eqn:Definition w H-1},
we define $w_z \in Y$ as the unique solution of the variational problem
\begin{equation}\label{eqn:normdefinition}
\int_Q \nabla_x w_z \cdot \nabla_x v \, dx \, dt = 
\langle z , v \rangle_Q \quad \forall v \in Y ,
\end{equation}
to conclude
\[
\|z\|^2_{L^2(0,T;H^{-1}(\Omega))} = \| \nabla_x w_z \|^2_{L^2(Q)} = 
\langle z , w_z \rangle_Q.
\]
Now, using standard arguments,  we can write the first-order optimality system
as the primal problem 
\begin{equation*}
\partial_t u-\Delta_x u = z \textup{ in } Q, \quad
u = 0 \textup{ on } \Sigma,\quad
u = 0 \textup{ on } \Sigma_0,
\end{equation*}
the adjoint problem
\begin{equation}
\label{eqn:adjointequation}
- \partial_t p -\Delta_x p = u - u_d \textup{ in } Q, \quad
p = 0 \textup { on } \Sigma , \quad
p = 0 \textup { on } \Sigma_T,
\end{equation}
and the gradient equation
\begin{equation}\label{eqn:gradientequation H-1}
p+\varrho w_z=0\textup{ in } Q. 
\end{equation}
Using the variational formulation (\ref{eqn:stateequation}) of the primal
problem, inserting the definition (\ref{eqn:normdefinition}), and
the gradient equation (\ref{eqn:gradientequation H-1}), we get
a first  variational equation to find $(u,p) \in X \times Y$ such that
\[    
\frac{1}{\varrho} \int_Q \nabla_x p \cdot \nabla_x v \, dx \, dt + 
\int_Q \Big[ \partial_t u \, v +\nabla_x u \cdot \nabla_x v
\Big] \, dx \, dt  = 0, \quad \forall v \in Y.
\]
On the other hand, when considering the variational formulation of the
adjoint problem and integrating by parts also in time, we arrive at the
second variational equation
\[
- \int_Q \Big[ p \, \partial_t q + \nabla_x p \cdot \nabla_x q \Big] 
\, dx \, dt 
+ \int_Q u \, q \, dx \, dt  = \int_Q u_d \, q \, dx \, dt, \quad
\forall q \in X .
\]
Hence, we end up with a variational problem to find
$(u,p) \in X \times Y$ such that
\begin{equation}\label{eqn:VF Optimality H-1}
{\mathcal{B}}(u,p;v,q) = \langle u_d , q \rangle_{L^2(Q)} , \quad 
\forall (v,q) \in Y \times X \, ,
\end{equation}
with
\begin{equation}\label{eqn:def bilinear form B H-1}
{\mathcal{B}}(u,p;v,q) := \frac{1}{\varrho} \, 
a(p,v) + b(u,v) - b(q,p) + c(u,q).
\end{equation}
Here\fredi{,} the bilinear form $b(\cdot,\cdot)$ is the same as used 
in Section~\ref{sec:prelim},
\[
a(p,v) := \int_Q \nabla_x p \cdot \nabla_x v \, dx \, dt, 
\quad \mbox{and}\quad 
c(u,q) := \int_Q u \, q \, dx \, dt .
\]
Note that $a(p,p) = \| p \|^2_Y$ and $c(u,u) = \| u \|^2_{L^2(Q)}$.

\begin{theorem}
For $u_d \in X^*$, the variational problem \eqref{eqn:VF Optimality H-1} 
admits a unique solution $(u,p) \in X \times Y$ satisfying the a priori 
estimates
\[
\| u \|_X \leq \frac{8}{\varrho} \, \| u_d \|_{X^*}, \quad
\| p \|_Y \leq \sqrt{2} \, 8 \, \| u_d \|_{X^*} .
\]
\end{theorem}
\begin{proof}
Using the Riesz representation theorem, we introduce 
operators $A : Y \to Y^*$, $B : X \to Y^*$, and $C: X \to X^*$, 
satisfying, for $u,q \in X$ and $p,v \in Y$,
\[
\langle A p , v \rangle_Q = a(p,v), \quad
\langle B u , v \rangle_Q = b(u,v), \quad
\langle C u , q \rangle_Q = c(u,q) .
\]
Hence, we can write the variational
problem \eqref{eqn:VF Optimality H-1} as operator equation
\[
\left( \begin{array}{cc} \frac{1}{\varrho} \, A & B \\[1mm] - B^* & C
\end{array} \right) \left( \begin{array}{c} p \\[1mm] u \end{array} \right)
= \left( \begin{array}{c} 0 \\[1mm] u_d \end{array} \right) .
\]
Since the operator $ A : Y \to Y^*$ is bounded and elliptic,
we can determine $p = - \varrho A^{-1} B u$ to obtain the Schur
complement system
\begin{equation}\label{eqn: H-1 Schur}
\Big[ C + \varrho B^* A^{-1} B \Big] u = u_d \quad \mbox{in} \; X^* .
\end{equation}
For $\overline{p} = A^{-1} B u$, we first have
\[
\| \overline{p} \|^2_Y = \int_Q |\nabla_x \overline{p}|^2 \, dx \, dt =
a(\overline{p},\overline{p}) = \langle A \overline{p} , \overline{p} \rangle_Q
= \langle B^* A^{-1} B u , u \rangle_Q \, .
\]
From the stability condition \eqref{eqn:infsup_b(.,.)} for the state 
equation, see Section~\ref{sec:prelim}, we immediately get
\[
\frac{1}{2\sqrt{2}} \, \| u \|_X \leq \sup\limits_{0 \neq v \in Y}
\frac{\langle B u , v \rangle_Q}{\| v \|_Y} =
\sup\limits_{0 \neq v \in Y}
\frac{\langle A \overline{p} , v \rangle_Q}{\| v \|_Y} \leq
\| \overline{p} \|_Y \, .
\]
Hence, we have
\[
\langle (C + \varrho B^* A^{-1} B) u , u \rangle_Q \geq 
\frac{1}{8} \, \varrho \, \| u \|^2_X \quad \mbox{for all} \; u \in X.
\] 
Thus, we conclude unique solvability of the Schur complement
system \eqref{eqn: H-1 Schur}, and from
\[
\frac{1}{8} \, \varrho \, \| u \|^2_X \leq
\langle (C + \varrho B^* A^{-1} B) u , u \rangle_Q =
\langle u_d , u \rangle_Q \leq \| u_d \|_{X^*} \| u \|_X
\]
we obtain the first estimate. Now, the boundedness of $B$,
i.e., the boundedness \eqref{eqn:boundedness_b(.,.)} of the bilinear 
form $b(\cdot,\cdot)$ yields
\[
\| p \|^2_Y = a(p,p) = - \varrho \, b(u,p) \leq 
\sqrt{2} \, \varrho \, \| u \|_X \| p \|_Y,
\]
i.e.,
\[
\| p \|_Y \leq \sqrt{2} \, \varrho \, \|u \|_X \leq 
\sqrt{2} \, 8 \, \| u_d \|_{X^*} .
\]
\end{proof}

\noindent
Although unique solvability of the variational problem 
\eqref{eqn:VF Optimality H-1} already implies a related stability 
condition for the bilinear form ${\mathcal{B}}(u,p;v,q)$, we
will present an alternative proof for this stability condition
in order to be able to derive related results for the
Galerkin discretization of \eqref{eqn:VF Optimality H-1}.

\begin{lemma}\label{thm:inf sup H-1}
The bilinear form \eqref{eqn:def bilinear form B H-1} satisfies the
stability condition
\begin{equation}\label{eqn:stability condition H-1}
\frac{1}{16} \, \varrho \, \Big[  \| u \|_X^2 + \| p \|_Y^2 \Big]^{1/2} \leq
\sup\limits_{0 \neq (v,q) \in Y \times X}
\frac{{\mathcal{B}}(u,p;v,q)}{\Big[  \| v \|_Y^2 + \| q \|_X^2 \Big]^{1/2}}
\end{equation}
for all $(u,p) \in X \times Y$, when assuming $\varrho \leq 1$.
\end{lemma}
\begin{proof}
For $u \in X \subset Y$, let $w_u \in Y$ be the unique solution of 
the variational
problem \eqref{eqn:Definition w H-1}. For $p \in Y$ and arbitrary 
$\alpha \in {\mathbb{R}}_+$, we have $v := u + w_u + \alpha p \in Y$. 
Choosing $q := \alpha u \in X$, we obtain
\begin{eqnarray*}
{\mathcal{B}}(u,p;v,q) & = & 
\frac{1}{\varrho} \,
a(p,u+w_u+\alpha p) + b(u,u+w_u+\alpha p) - b(\alpha u,p) + c(u,\alpha u) 
\\ & \geq & \frac{\alpha}{\varrho} \, 
\| p \|^2_Y + \frac{1}{\varrho} \, a(p,u+w_u) + b(u,u+w_u) \, .
\end{eqnarray*}
Following the proof of \cite[Theorem 2.1]{OS15}, we use 
\begin{eqnarray*}
b(u,u+w_u) & = & \int_Q \Big[ \partial_t u \, (u+w_u) +
\nabla_x u \cdot \nabla_x (u+w_u) \Big] dx \, dt \\ & & \hspace*{-2.5cm} = \,
\frac{1}{2} \, \| u(T) \|^2_{L^2(\Omega)} + 
\int_Q \nabla_x w_u \cdot \nabla_x w_u \, dx \, dt +
\| u \|^2_Y + \int_Q \nabla_x u \cdot \nabla_x w_u \, dx \, dt \\
& & \hspace*{-2.5cm} \geq \,
\| w_u \|^2_Y + \| u \|_Y^2 - \| u \|^2_Y \| w_u \|_Y^2 \\ 
& & \hspace*{-2.5cm} \geq \, 
\frac{1}{2} \, \Big[ \| w_u \|_Y^2 + \| u \|^2_Y \Big] = 
\frac{1}{2} \, \| u \|^2_X .
\end{eqnarray*}
Moreover, for $\gamma \in {\mathbb{R}}_+$, we have
\begin{eqnarray*}
a(p,u+w_u) & = & \int_Q \nabla_x p \cdot \nabla_x (u+w_u) \, dx \, dt
\, \geq \, - \| p \|_Y \| u+w_u \|_Y \\ & \geq &
- \frac{1}{2\gamma} \, \| p \|^2_Y - \frac{1}{2}\gamma \, \| u+w_u \|^2_Y
\\ & \geq &
- \frac{1}{2\gamma} \, \| p \|^2_Y 
- \gamma \, \Big( \| u \|_Y^2 + \| w_u \|^2_Y \Big) =
- \frac{1}{2\gamma} \, \| p \|^2_Y - \gamma \, \| u \|_X^2 .
\end{eqnarray*}
Choosing $\gamma = \frac{1}{4}\varrho$ and
$\alpha = \frac{1}{4}\varrho + \frac{2}{\varrho}$, we get
\begin{eqnarray*}
{\mathcal{B}}(u,p;v,q) & \geq & 
\frac{1}{\varrho} \left( \alpha - \frac{1}{2\gamma} \right) \, \| p \|^2_Y 
+ \left( \frac{1}{2} - \frac{\gamma}{\varrho} \right) \, 
\| u \|_X^2 \\ & = & 
\frac{1}{4} \, \Big[ \| u \|^2_X + \| p \|^2_Y \Big] \, .
\end{eqnarray*}
On the other hand, we have
\begin{eqnarray*}
\| q \|_X^2 + \| v \|_Y^2 & = & 
\alpha^2 \, \| u \|_X^2 + \| u + w_u + \alpha p \|_Y^2 
\\ & \leq &
\alpha^2 \, \| u \|^2_X + 
\Big( \| u \|_Y + \| w_u \|_Y + \alpha \, \| p \|_Y \Big)^2 
\\ & \leq &
\alpha^2 \, \| u \|^2_X +   
(2+\alpha^2) \, \Big( \| u \|^2_Y + \| w_u \|^2_Y + \| p \|_Y^2
\Big) \\ & \leq &
2 ( 1+ \alpha^2) \, \Big[ \| u \|_X^2 + \| p \|_Y^2 \Big] 
\\ & = &
2 \Big( 2 + \frac{1}{16} \varrho^2 + \frac{4}{\varrho^2} \Big) \, 
\Big[ \| u \|_X^2 + \| p \|_Y^2 \Big] 
\\ & \leq &
\Big( 4 + \frac{1}{8} + 8  \Big) \, 
\frac{1}{\varrho^2} \Big[ \| u \|_X^2 + \| p \|_Y^2 \Big] \leq 
\frac{16}{\varrho^2} \Big[ \| u \|_X^2 + \| p \|_Y^2 \Big] ,
\end{eqnarray*}
provided that $\varrho \leq 1$. Therefore,
\[
{\mathcal{B}}(u,p;v,q) \geq \frac{1}{4} \,
\Big[ \| u \|_X^2 + \| p \|_Y^2 \Big] \geq
\frac{1}{16} \, \varrho \,  \Big[ \| u \|_X^2 + \| p \|_Y^2 \Big]^{1/2}
\Big[ \| q \|_X^2 + \| v \|_Y^2 \Big]^{1/2}
\]
follows. This concludes the proof.
\end{proof}

\section{Discretization}\label{Section:Discretization}
As before, let $X_{0,h} \subset X_0$ and $Y_h \subset Y$ be some conforming
space-time finite element spaces satisfying $X_{0,h} \subseteq Y_h$.
Again, we choose $X_{0,h} = S_h^1(Q_h) \cap X_0$, but now we use 
$Y_h = S_h^1(Q) \cap Y$. By construction, we have $X_{0,h} \subset Y_h$.

Instead of (\ref{eqn:Definition w H-1 FEM}), we now consider the
variational formulation to find $w_{u,h} \in Y_h$ such that
\begin{equation}\label{eqn:Definition w H-1 FEM Yh}
\int_Q \nabla_x w_h \cdot \nabla_x v_h \, dx \, dt =
\int_Q \partial_t u \, v_h \, dx \, dt, \quad \forall v_h \in Y_h,
\end{equation}
to define the discrete norm
\[
\| u \|_{X_{0,h}} := \Big[ \| w_{u,h} \|_Y^2 + \| u \|_Y^2 \Big]^{1/2} .
\]
The space-time finite element discretization of the variational
formulation (\ref{eqn:VF Optimality H-1}) is to find
$(u_h,p_h) \in X_{0,h} \times Y_h$ such that
\begin{equation}\label{eqn:VF Optimality H-1 FEM}
{\mathcal{B}}(u_h,p_h;v_h,q_h) = \langle u_d , q_h \rangle_{L^2(Q)} , \quad 
\forall (v_h,q_h) \in Y_h \times X_{0,h} \, .
\end{equation}
As in the continuous case, see Theorem \ref{thm:inf sup H-1},
and following \cite[Section 3]{OS15}, we can confirm a discrete
inf-sup condition for the bilinear form 
${\mathcal{B}}(\cdot,\cdot;\cdot,\cdot)$.

\begin{lemma}
\label{thm:inf sup H-1 FEM}
The bilinear form (\ref{eqn:def bilinear form B H-1}) 
satisfies the discrete stability condition
\begin{equation}\label{eqn:stability condition H-1 FEM}
\frac{1}{16} \, \varrho \, 
\Big[  \| u_h \|_{X_{0,h}}^2 + \| p_h \|_Y^2 \Big]^{1/2} \leq
\sup\limits_{0 \neq (v_h,q_h) \in Y_h \times X_{0,h}}
\frac{{\mathcal{B}}(u_h,p_h;v_h,q_h)}{\Big[  \| v_h \|_Y^2 + 
\| q_h \|_{X_{0,h}}^2 \Big]^{1/2}}
\end{equation}
for all $(u_h,p_h) \in X_{0,h} \times Y_h$, when assuming 
$X_{0,h} \subseteq Y_h$ and $\varrho \leq 1$.
\end{lemma}
\begin{proof}
Since the proof follows the lines of the proof of Theorem
\ref{thm:inf sup H-1}, we only sketch the most important steps.

For $u_h \in X_{0,h}$, let $w_{u_h,h} \in Y_h$ be the unique finite element solution
of the variational problem (\ref{eqn:Definition w H-1 FEM Yh}). For
$p_h \in Y_h$, and due to $X_{0,h} \subset Y_h$, we then have
$v_h := u_h + w_{u_h,h} + \alpha p_h \in Y_h$, $\alpha \in {\mathbb{R}}_+$.
Moreover, set $q_h = \alpha u_h \in X_{0,h}$.

As in the proof of Theorem \ref{thm:inf sup H-1}, we now conclude
\[
{\mathcal{B}}(u_h,p_h;v_h,q_h) \geq \frac{1}{4} \, \Big[
\| p_h \|^2_Y + \| w_{u_h,h} \|_Y^2 + \| u_h \|_Y^2 \Big] =
\frac{1}{4} \, \Big[ \| u_h \|_{X_{0,h}}^2 + \| p_h \|^2_Y \Big] .
\]
On the other hand, and as in the proof of Theorem \ref{thm:inf sup H-1}, 
we have
\begin{eqnarray*}
\| q_h \|_{X_{0,h}}^2 + \| v_h \|_Y^2 & = & 
\alpha^2 \, \| u_h \|_{X_{0,h}}^2 + \| u_h + w_{u_h,h} + \alpha p_h \|_Y^2 
\\ & \leq &
\frac{16}{\varrho^2} \Big[ \| u_h \|_{X_{0,h}}^2 + \| p_h \|_Y^2 \Big] .
\end{eqnarray*}
Now the assertion follows as in the continuous case.
\end{proof}

\noindent
The discrete inf-sup condition (\ref{eqn:stability condition H-1 FEM})
implies unique solvability of the space-time finite element scheme
(\ref{eqn:VF Optimality H-1 FEM}). By combining (\ref{eqn:VF Optimality H-1
  FEM}) with (\ref{eqn:VF Optimality H-1}) 
and by using the inclusions $X_{0,h} \subset X_0$ and
$Y_h \subset Y$, we also conclude the Galerkin orthogonality
\begin{equation}\label{eqn:VF Optimality H-1 Galerkin}
{\mathcal{B}}(u-u_h,p-p_h;v_h, q_h) = 0,\quad \forall 
(v_h,q_h) \in Y_h \times X_{0,h}.
\end{equation}

\begin{theorem}\label{thm:h1conv}
Let $(u,p) \in X_0 \times Y$ and $(u_h,p_h) \in X_{0,h} \times Y_h$ 
be the unique solutions of the variational problems 
(\ref{eqn:VF Optimality H-1}) and (\ref{eqn:VF Optimality H-1 FEM}),
respectively, where  $X_{0,h} \subset X_0$ and $Y_h \subset Y$.
Furthermore, assume that $(u,p) \in H^2(Q) \times H^2(Q)$.
Then there holds the discretization error estimate
\[
\varrho \, \Big[  \| u - u_h \|_{X_{0,h}}^2 + 
\| p - p_h \|_Y^2 \Big]^{1/2} 
\leq c \, h \, \left[ 
\frac{1}{\varrho} \, \|p\|_{H^2(Q)} + \|u\|_{H^2(Q)} \right] \, .
\]
\end{theorem}

\begin{proof}
For arbitrary $(z_h,r_h) \in X_{0,h} \times Y_h$, 
the discrete inf-sup condition (\ref{eqn:stability condition H-1 FEM}) 
and the Galerkin orthogonality (\ref{eqn:VF Optimality H-1 Galerkin})
immediately yield the estimates
\begin{eqnarray*}
\frac{1}{16} \, \varrho \, \Big[  \| u_h - z_h \|_{X_{0,h}}^2 + 
\| p_h - r_h \|_Y^2 \Big]^{1/2} 
&  \\ && \hspace*{-4cm} \leq \,
\sup\limits_{0 \neq (v_h,q_h) \in Y_h \times X_{0,h}}
\frac{{\mathcal{B}}(u_h-z_h,p_h-r_h;v_h,q_h)}{[  \| v_h \|_Y^2 + 
\| q_h \|_{X_{0,h}}^2 ]^{1/2}} \\ && \hspace*{-4cm} = \,
\sup\limits_{0 \neq (v_h,q_h) \in Y_h \times X_{0,h}}
\frac{{\mathcal{B}}(u-z_h,p-r_h;v_h,q_h)}{[  \| v_h \|_Y^2 + 
\| q_h \|_{X_{0,h}}^2 ]^{1/2}}.
\end{eqnarray*}
Now we consider
\begin{eqnarray*}
\frac{{\mathcal{B}}(u-z_h,p-r_h;v_h,q_h)}{[  \| v_h \|_Y^2 + 
\| q_h \|_{X_{0,h}}^2 ]^{1/2}} &  \\ && \hspace*{-4cm} = \,
\frac{\frac{1}
{\varrho}a(p-r_h,v_h) + b(u-z_h,v_h) - b(q_h,p-r_h) + c(u-z_h,q_h) }
{[  \| v_h \|_Y^2 + \| q_h \|_{X_{0,h}}^2 ]^{1/2}} \\ && \hspace*{-4cm} \leq \,
\frac{\frac{1}{\varrho}a(p-r_h,v_h) + b(u-z_h,v_h)}{\| v_h \|_Y}
- \frac{b(q_h,p-r_h)}{[\| v_h \|_Y^2 + \| q_h \|_{X_{0,h}}^2 ]^{1/2}}
+
\frac{c(u-z_h,q_h) }{\| q_h \|_{X_{0,h}}} \\ & & \hspace*{-4cm} \leq \,
\frac{1}{\varrho} \, \| p - r_h \|_Y + \sqrt{2} \, \| u - z_h \|_{X_0} +
c \, \| u - z_h \|_{L^2(Q)} 
- \frac{b(q_h,p-r_h)}{[\| v_h \|_Y^2 + \| q_h \|_{X_{0,h}}^2 ]^{1/2}} .
\end{eqnarray*}
Integrating by parts in time and using $q_h=0$ in $\Sigma_0$, we obtain
\begin{eqnarray*}
b(q_h,p-r_h) & = &
\int_Q \Big[
\partial_t q_h \, (p - r_h) + \nabla_x q_h \cdot \nabla_x (p-r_h)
\Big] \, dx \, dt \\ & & \hspace*{-2.7cm} = \,
\int_\Omega q_h(T) (p(T)-r_h(T)) dx +
\int_Q \Big[
- q_h \, \partial_t (p - r_h) + \nabla_x q_h \cdot \nabla_x (p-r_h)
\Big] dx \, dt \\ & & \hspace*{-2.7cm} \leq \,
\| q_h(T) \|_{L^2(\Omega)} \| p(T)-r_h(T) \|_{L^2(\Omega)} +
\sqrt{2} \| q_h \|_Y \Big[
\| \partial_t(p-r_h) \|_{Y^*} + \| p - r_h \|_Y
\Big].
\end{eqnarray*}
Therefore, the inequalities
\begin{eqnarray*}
\frac{b(q_h,p-r_h)}{[\| v_h \|_Y^2 + \| q_h \|_{X_{0,h}}^2 ]^{1/2}} & \leq &
 \frac{b(q_h,p-r_h)}{\| q_h \|_Y } \\ & & \hspace*{-4cm} \leq \, 
\frac{\| q_h(T) \|_{L^2(\Omega)}}{\| q_h \|_Y} \| p(T)-r_h(T) \|_{L^2(\Omega)} +
\sqrt{2} \Big[
\| \partial_t(p-r_h) \|_{Y^*} + \| p - r_h \|_Y
\Big]
\end{eqnarray*}
follow. Let $\tau_\ell^\mu \subset {\mathbb{R}^\mu}$ with $\mu=d$ or $\mu=d+1$
be a shape regular simplicial finite element with mesh size $h_\ell$. 
For a piecewise linear finite element function, we then have the equivalence
\[
\int_{\tau_\ell^\mu} [v_h(x)]^2 \, dx \simeq 
h_\ell^\mu \sum\limits_{k=1}^{\mu+1} v_{\ell_k}^2,
\]
where the $v_{\ell_k}$ are the local nodal values of $v_h$.
Hence, we can write
\[
\| q_h(T) \|^2_{L^2(\Omega)} = \sum\limits_{\tau_\ell^n \in \Sigma_T} 
\| q_h(T) \|^2_{L^2(\tau_\ell^n)} \simeq \sum\limits_{\tau_\ell^n \in \Sigma_T}
h_\ell^n \sum\limits_{k=1}^{n+1} v_{\ell_k}^2 
\]
as well as
\[
\| q_h \|^2_{L^2(Q)} = \sum\limits_{\tau_\ell^{n+1} \in Q} 
\| v_h \|^2_{L^2(\tau_\ell^{n+1})} \simeq \sum\limits_{\tau_\ell^{n+1} \in Q} 
h_\ell^{n+1} \sum\limits_{k=1}^{n+2} v_{\ell_k}^2 \, . 
\]
Thus, we conclude
\[
\| q_h(T) \|^2_{L^2(\Omega)} \leq c \, h_T^{-1} \, \| q_h \|^2_{L^2(Q)}, 
\]
where $h_T \simeq h_\ell$ is the globally quasi-uniform mesh size 
of all space-time finite elements sharing $\Sigma_T$. Since we have
$q_h =0$ on $\Sigma$, we finally obtain
\[
\| q_h(T) \|_{L^2(\Omega)} \leq c \, h_T^{-1/2} \, \| \nabla_x q_h \|_{L^2(Q)} 
= c \, h_T^{-1/2} \, \| q_h \|_Y \, .
\]
When summarizing all the previous steps, this gives
\begin{eqnarray*}
\frac{1}{16} \, \varrho \, \Big[  \| u_h - z_h \|_{X_{0,h}}^2 + 
\| p_h - r_h \|_Y^2 \Big]^{1/2} 
&  \\ && \hspace*{-6cm} \leq \,
\Big( \sqrt{2} + \frac{1}{\varrho} \Big) \, 
\| p - r_h \|_Y + \sqrt{2} \, \| u - z_h \|_{X_0} +
c \, \| u - z_h \|_{L^2(Q)} \\
&& \hspace*{-5cm}
+ c \, h_T^{-1/2} \, \| p(T)-r_h(T) \|_{L^2(\Omega)} +
\sqrt{2} \, \| \partial_t(p-r_h) \|_{Y^*} \, .
\end{eqnarray*}
As in \cite{OS15}, we may now chose $z_h = P_h u \in X_{0,h}$ and 
$r_h = \overline{P}_h p \in Y_h$ being the related $H^1(Q)$ projections.
Using standard arguments, see, e.g. the proof of Theorem 3.3 in
\cite{OS15}, we finally obtain
\[
\frac{1}{16} \, \varrho \, \Big[  \| u_h - z_h \|_{X_{0,h}}^2 + 
\| p_h - r_h \|_Y^2 \Big]^{1/2} 
\leq c \, h \, \left[ 
\frac{1}{\varrho} \, \|p\|_{H^2(Q)} + \|u\|_{H^2(Q)} \right] \, .
\]
The assertion now follows when applying the triangle inequality
and once again the approximation properties in $X_{0,h}$ and $Y_h$,
respectively.
\end{proof}

\section{Numerical results}\label{Section:Numerical Examples}
In our numerical experiments, we consider examples in both two and three
space dimensions. In the two-dimensional case, we consider 
$\Omega=(0,1)^2$, $T=1$, and therefore $Q=(0,1)^3$. The coarsest space-time 
mesh contains $125$ vertices and $384$ tetrahedral elements with the mesh 
size $h=1/4$. By a uniform red-green refinement \cite{JB95}, we reduce 
the mesh size recursively, i.e., $h=1/8$, $1/16$ and so on. In this case,
the numerical examples are tested on a desktop with Intel@ Xeon@ Processor
E5-1650 v4 ($15$ MB Cache, $3.60$ GHz), and $64$ GB memory.   

In three space dimensions, we set $\Omega=(0, 1)^3$, $T=1$, and
therefore $Q=(0, 1)^4$. We start from an initial mesh containing $178$
vertices and $960$ pentatopes with a mesh size $h\approx 1$. Following a
bisection approach \cite{RS08}, we perform a sequence of uniform refinements 
of the initial mesh. In this case, the numerical examples are tested on
a compute node with two $20$-core Intel Broadwell Processors (Xeon E5-2698v4,
$2.2$ Ghz) and $1$ TB memory. 

For the solution of the discrete first-order optimality system,
we use an algebraic multigrid preconditioned GMRES method with a relative 
residual error reduction $\varepsilon=10^{-8}$ as a stopping criterion.
We refer to \cite{OSHY19} for more details on
constructing the algebraic multigrid preconditioner and the performance study
for solving such a coupled system.  

\subsection{An example with explicitly known solution}\label{sec:2dexm}
In order to check the convergence rates, we first consider an example with
an explicitly known solution of the first-order optimality system,
i.e., for $d=2$,
\begin{equation*}
  \begin{aligned}
    u(x,t)&=2\pi^2\sin(\pi x_1)\sin(\pi x_2)\left(ct^2+t\right),\\
    p(x,t)&=-\varrho\sin(\pi x_1)\sin(\pi x_2)\left( at^2+bt+1\right),\\
    z(x,t)&=2\pi^2\sin(\pi x_1)\sin(\pi x_2)\left( at^2+bt+1\right),
  \end{aligned}
\end{equation*}
where 
\[
a=-\frac{4\pi^4+2\pi^2}{2\pi^2+2}, \quad
b=\frac{4\pi^4-2}{2\pi^2+2}, \quad
c=-\frac{2\pi^2+1}{2\pi^2+2}.
\]
The regularization parameter is set to $\varrho=0.01$. By definition,
$u$ fulfills the homogeneous initial and boundary conditions for the 
state equation, while $p$ satisfies the homogeneous terminal and 
boundary conditions for the adjoint equation, see the illustration 
in Fig.~\ref{fig:energysol_p21}. The numerical results are given in
Table \ref{tab:eocl2h1_p21}, where we present the errors for the 
approximate solutions $u_h$ and $p_h$ in $Y=L^2(0,T; H^1_0(\Omega))$.
The estimated order of convergence (eoc) corresponds 
to the estimate as given in Theorem \ref{thm:h1conv}.
Further, we observe a nearly optimal concergence rate in
$L^2(Q)$, see Table \ref{tab:eocl2l2_p21}.
Finally, a second-order convergence rate of the objective functional
is observed, see Table \ref{tab:l2_eocobj_p21}. 

\begin{figure}
  \centering
  \includegraphics[width=0.32\textwidth]{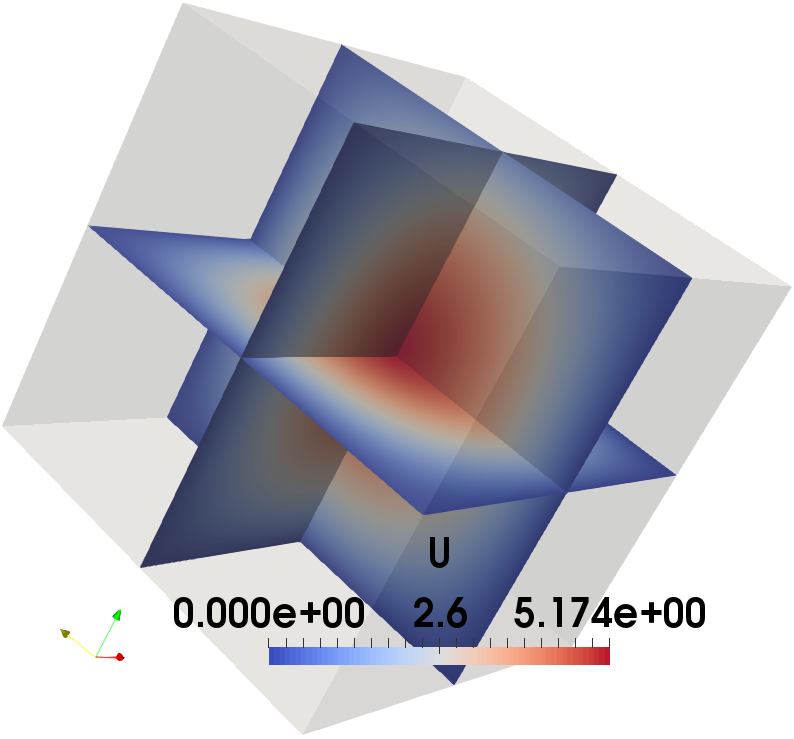}
  \includegraphics[width=0.32\textwidth]{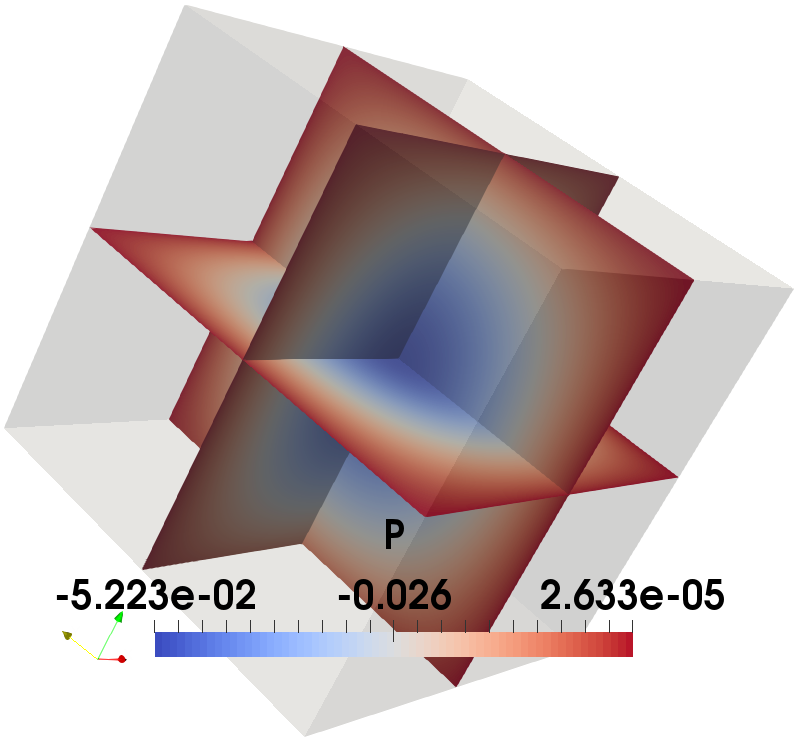}
  \includegraphics[width=0.32\textwidth]{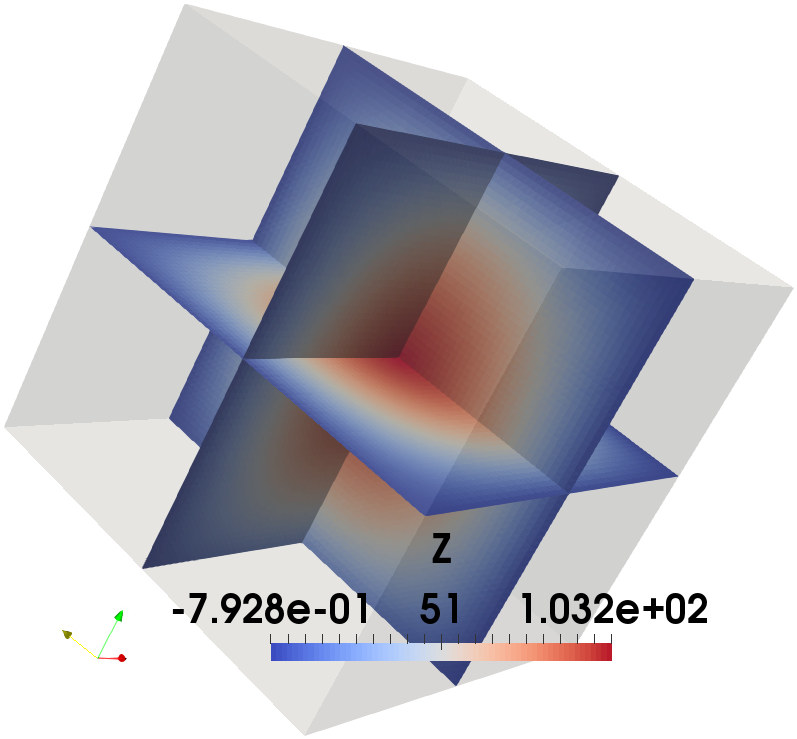}
  \caption{Example 1, numerical solutions of $u$, $p$, and $z$ for the linear
    model problem using energy regularization.}
  \label{fig:energysol_p21}
\end{figure} 

\begin{table}[h]\caption{Example 1, estimated order of convergence (eoc) 
of $u_h$ and $p_h$ in $Y$.} 
\centering
\begin{tabular}{rrrrrr}
\toprule
\#Dofs & $h$ & $\|u-u_h\|_Y$ & eoc & $\|p-p_h\|_Y$ & eoc  \\
\midrule
$250$ & $1/4$& $4.611e-0$    & $-$  & $4.651e-2$    & $-$       \\
$1,458$ & $1/8$& $2.303e-0$    & $1.002$  &$2.313e-2$    & $1.008$  \\
$9,826$ & $1/16$ & $1.129e-0$    & $1.028$ & $1.134e-2$    & $1.028$  \\
$71,874$ & $1/32$ & $5.572e-1$    & $1.019$ & $5.601e-3$    & $1.018$   \\
$549,250$ & $1/64$ &$2.766e-1$ & $1.010$ & $2.783e-3$ & $1.009$ \\
$2,146,689$ & $1/128$ &$1.379e-1$    & $1.005$ & $1.388e-3$    & $1.004$  \\
\bottomrule
\end{tabular}\label{tab:eocl2h1_p21}
\end{table}

\begin{table}[h]\caption{Example 1, estimated order of convergence (eoc) 
of $u_h$ and $p_h$ in $L^2(Q)$.} 
\centering
\begin{tabular}{rrrrrr}
\toprule
\#Dofs &  $h$ & $\|u-u_h\|_{L^2(Q)}$ & eoc & $\|p-p_h\|_{L^2(Q)}$ & eoc   \\
\midrule
$250$ & $1/4$& $2.365e-1$    & $-$  & $2.829e-3$    & $-$   \\
$1,458$ & $1/8$& $5.685e-2$    & $2.057$  &$7.693e-4$    & $1.879$  \\
$9,826$ & $1/16$ & $1.410e-2$    & $2.011$ & $2.171e-4$   & $1.825$  \\
$71,874$ & $1/32$ & $3.642e-3$    & $1.953$ & $6.063e-5$    & $1.841$ \\
$549,250$ &  $1/64$ &$9.992e-4$    & $1.876$ & $1.655e-5$    & $1.873$ \\
$2,146,689$ &  $1/128$ &$2.759e-4$    & $1.847$ & $5.415e-6$    & $1.612$ \\
\bottomrule
\end{tabular}\label{tab:eocl2l2_p21}
\end{table}

\begin{table}[h]\caption{Example 1, $J(u_h,z_h)$, $|J(u_h,z_h)-J(u,z)|$,
$J(u,z)=4.53541e-1$.} 
\centering
\begin{tabular}{rrrrr}
\toprule
\#Dofs &  $h$ & $J(u_h, z_h)$ &  $|J(u_h, z_h)-J(u,z)|$ &  eoc   \\
\midrule
$250$ & $1/4$& $5.87348e-1$      & $1.3381e-1$    & $-$   \\
$1,458$ & $1/8$& $4.81288e-1$     &$2.7747e-2$    & $2.270$  \\
$9,826$ & $1/16$ & $4.59930e-1$    & $6.3890e-3$    & $2.119$  \\
$71,874$ & $1/32$ & $4.55054e-1$   & $1.5130e-3$    & $2.078$ \\
$549,250$ &  $1/64$ &$4.53863e-1$   & $3.2200e-4$    & $2.232$ \\
$2,146,689$ &  $1/128$ &$4.53557e-1$   & $1.6000e-5$    & $4.331$ \\
\bottomrule
\end{tabular}\label{tab:l2_eocobj_p21}
\end{table}

\subsection{An example with a discontinuous target}
As in the previous example, we have $\Omega = (0,1)^2$, $T=1$, i.e.,
$Q=(0,1)^3$, but now we consider the discontinuous target function
\begin{equation*}
  u_d(x,t) =
  \begin{cases}
    &1\quad\text{ if }
    \sqrt{(x_1-\frac{1}{2})^2+(x_2-\frac{1}{2})^2+(t-\frac{1}{2})^2}
    \leq\frac{1}{4},\\
    &0\quad \textup{ else.}
  \end{cases}
\end{equation*}
Here, the regularization parameter is set to $\varrho=10^{-4}$. Following the 
approach described in \cite{OSHY18}, we have used a residual based 
error indicator to drive an adaptive mesh refinement. The space-time finite 
element solutions for the state $u$ and the adjoint $p$ are 
provided in Fig. \ref{fig:ex4sol} in comparison with the time-dependent 
target $u_d$. The control $z$ is then reconstructed from
\eqref{eqn:normdefinition}, \eqref{eqn:adjointequation}, and 
\eqref{eqn:gradientequation H-1} by an $L^2$ projection on the space of 
element-wise constant functions.
More precisely, we look for an element-wise constant control $z_h$ such that 
\[
\langle z_h,\varphi_h\rangle_{L^2(Q)} = 
-\frac{1}{\varrho}\langle \partial_t p_h + u_h -u_d, \varphi_h \rangle_{L^2(Q)}
\]
holds for all element-wise constant test functions $\varphi_h$. The results are
given
in the last column of Fig. \ref{fig:ex4sol}. We clearly see that the control 
is concentrated near the interface, where the target exhibits a jump. The
adaptive mesh is illustrated in Fig. \ref{fig:ex4mesh} at the $78$th refining
step,  which contains $3,398,213$ grid points. The total number of degrees of
freedom for the coupled state and adjoint equation is $6,796,426$.   

\begin{figure}
  \centering
  \includegraphics[width=0.24\textwidth]{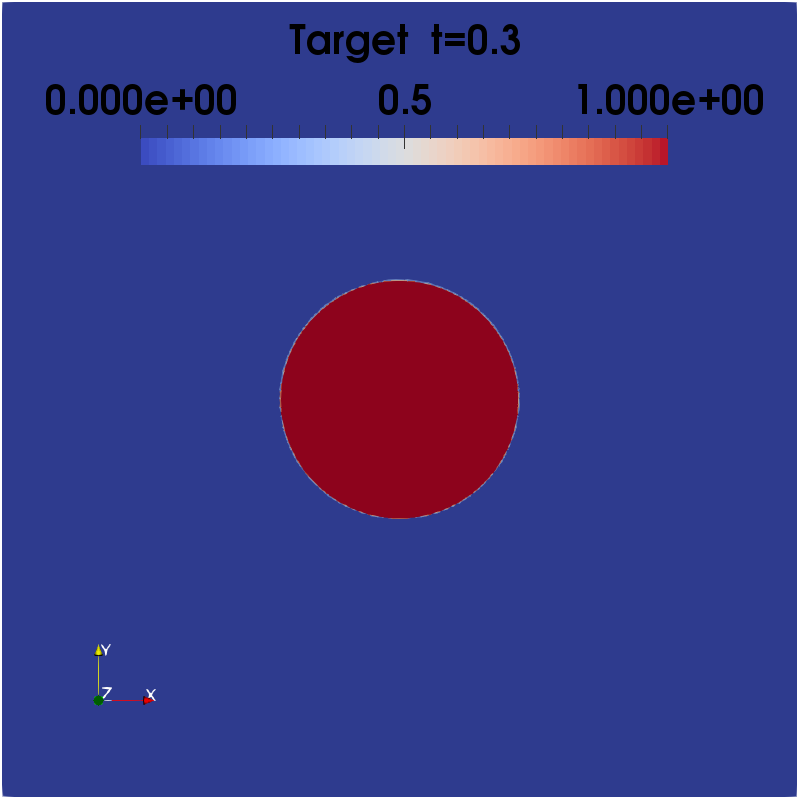}
  \includegraphics[width=0.24\textwidth]{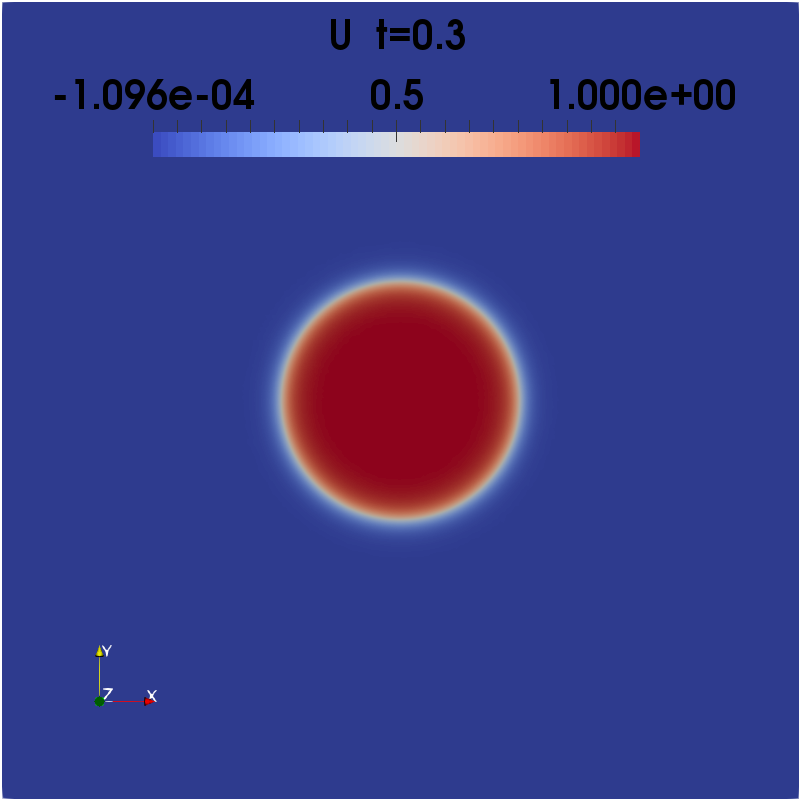}
  \includegraphics[width=0.24\textwidth]{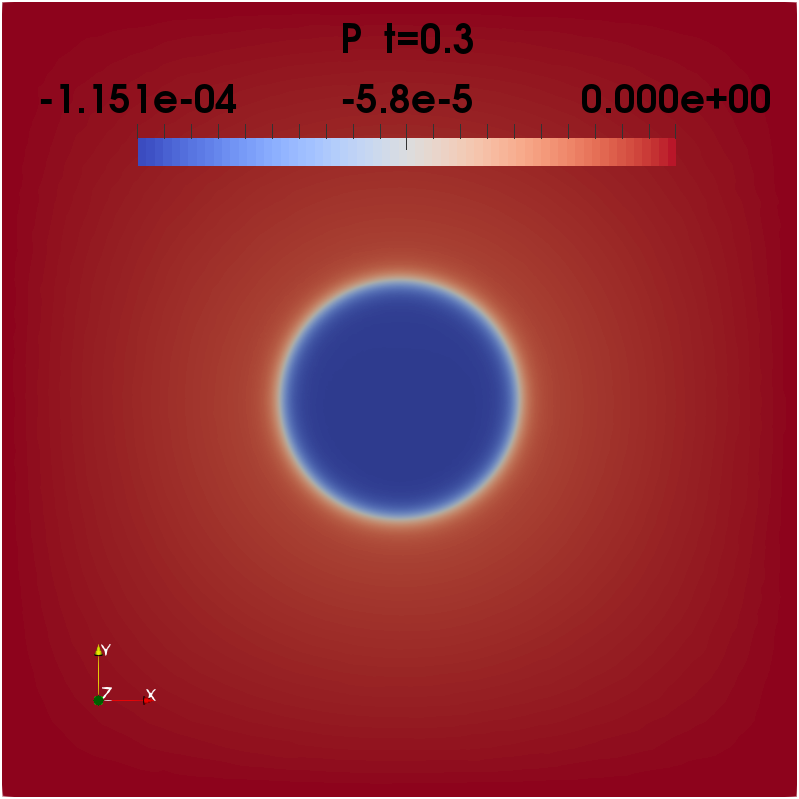}
  \includegraphics[width=0.24\textwidth]{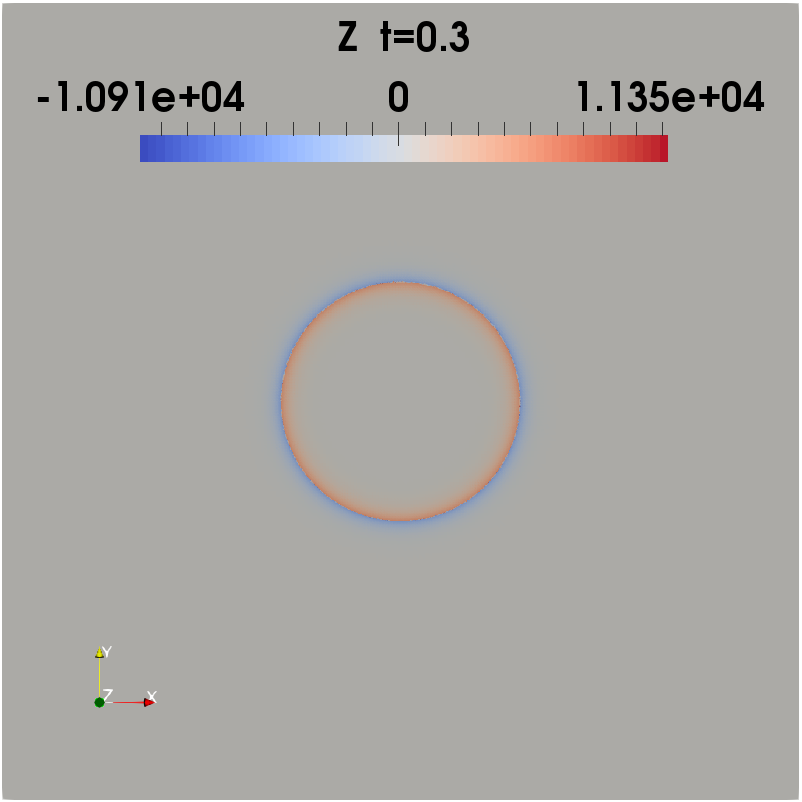}
  \includegraphics[width=0.24\textwidth]{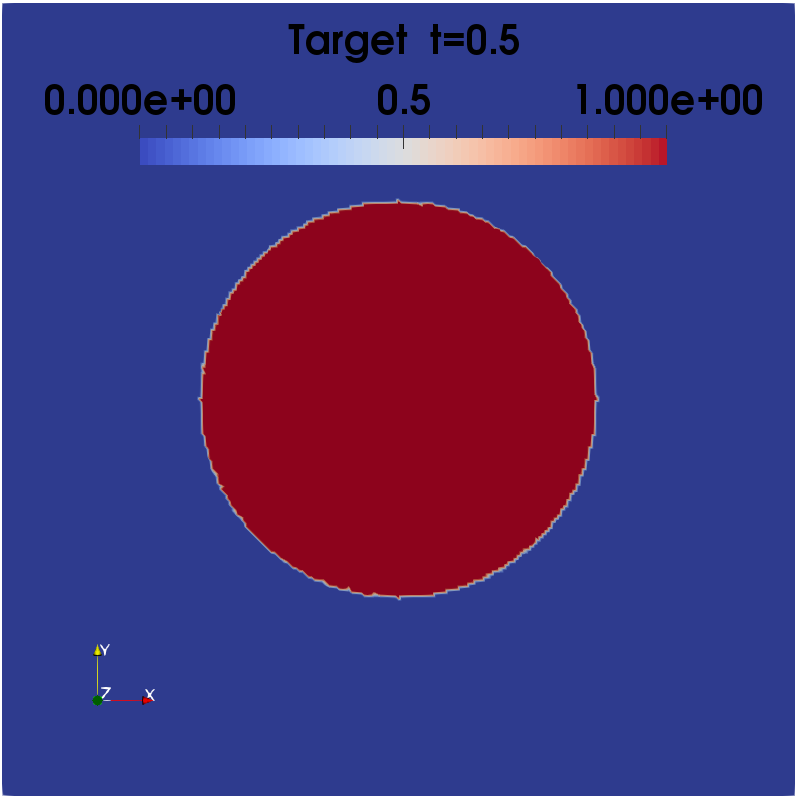}
  \includegraphics[width=0.24\textwidth]{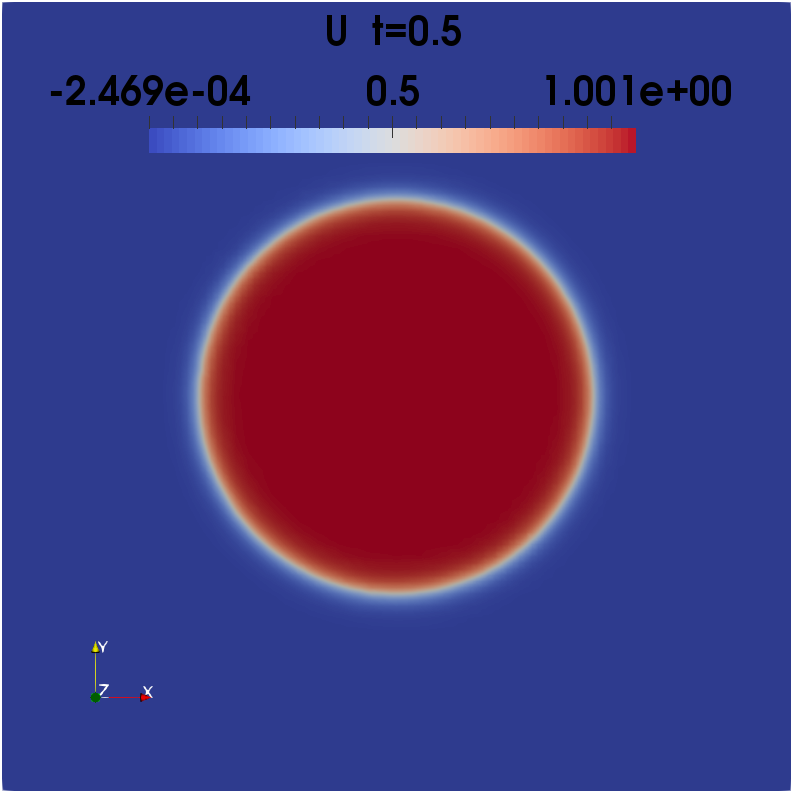}
  \includegraphics[width=0.24\textwidth]{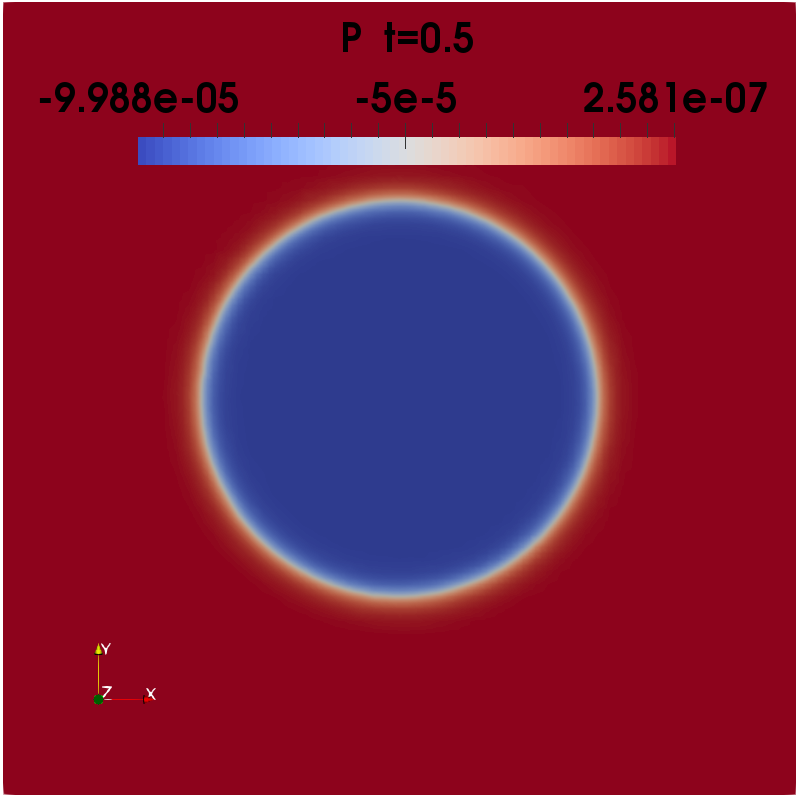}
  \includegraphics[width=0.24\textwidth]{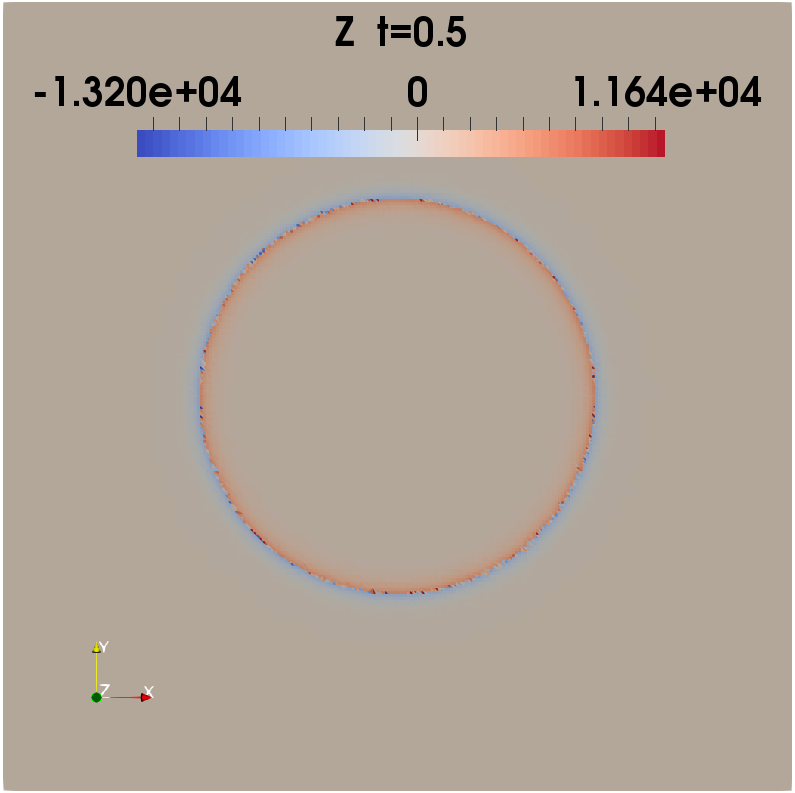}
  \includegraphics[width=0.24\textwidth]{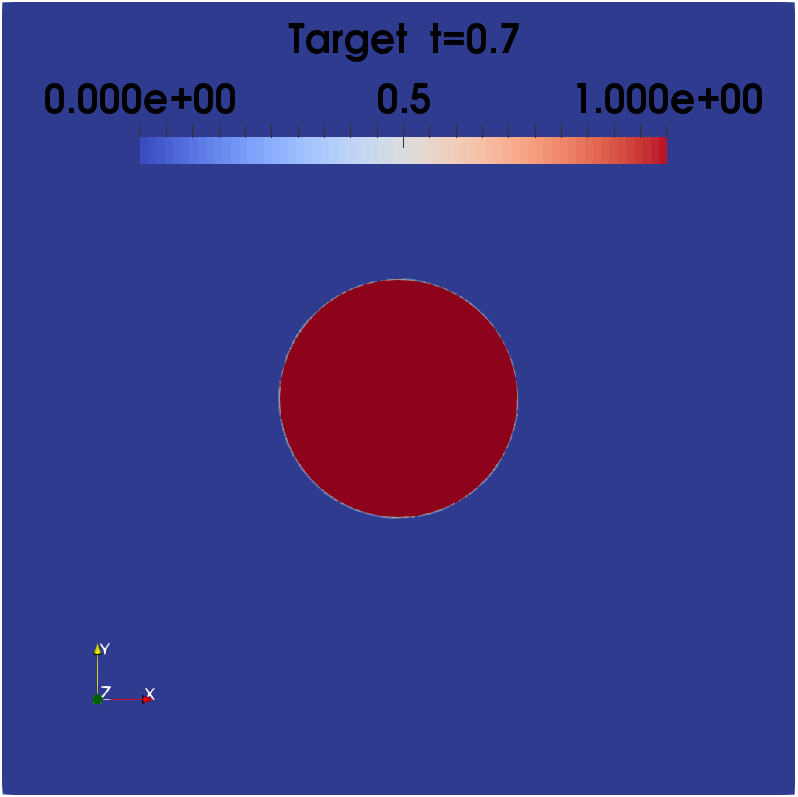}
  \includegraphics[width=0.24\textwidth]{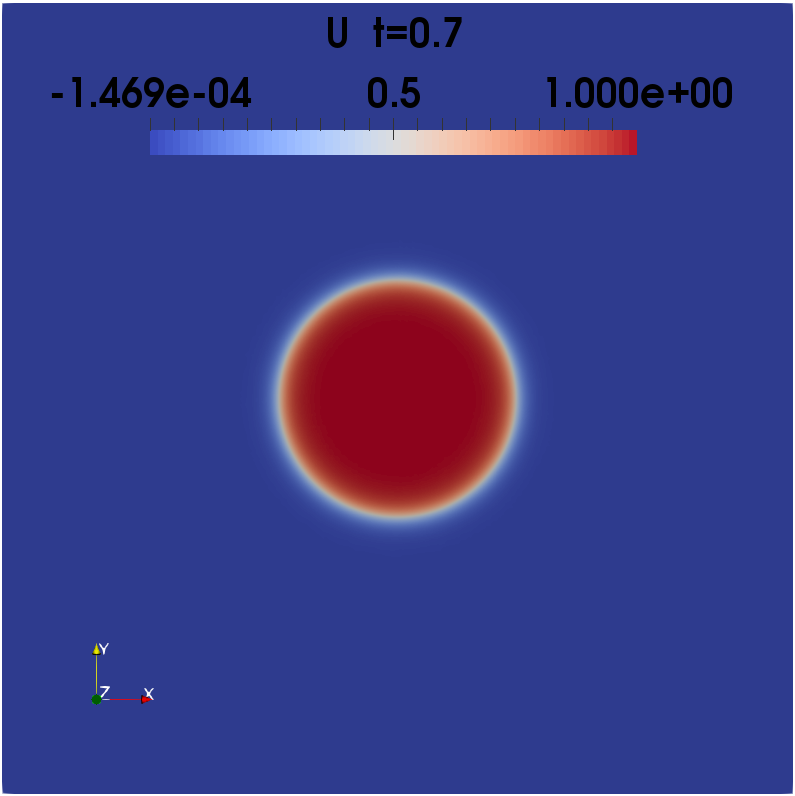}
  \includegraphics[width=0.24\textwidth]{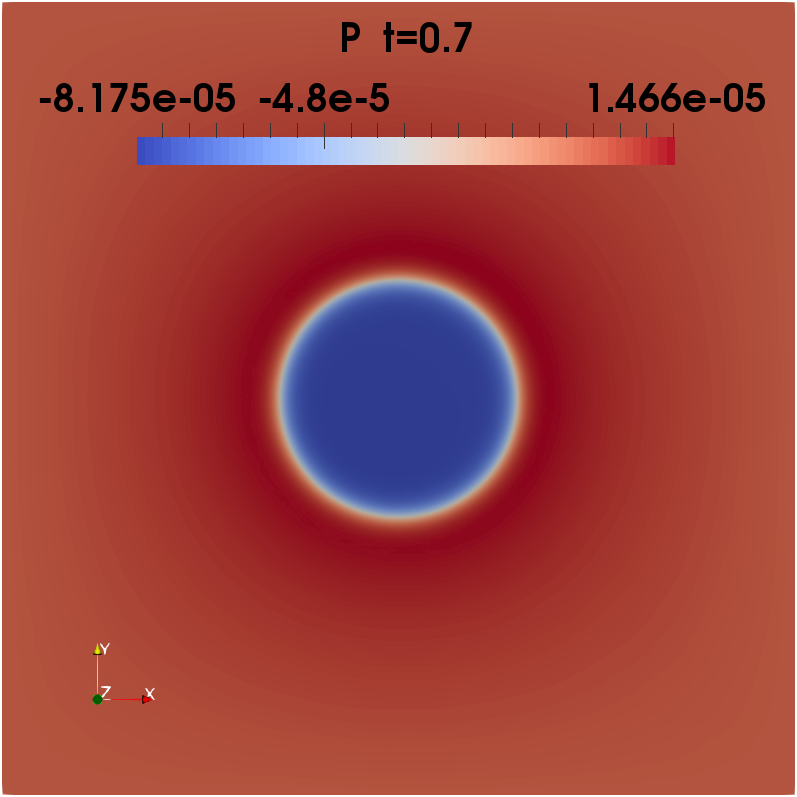}
  \includegraphics[width=0.24\textwidth]{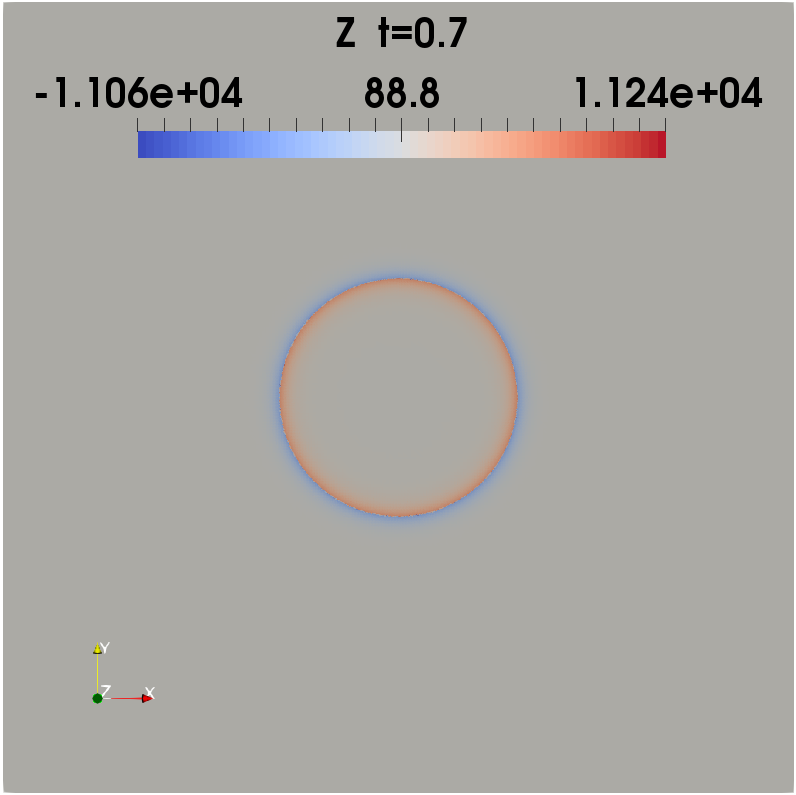}
  \caption{Example 2, Target $u_d$ and numerical solutions $u_h$, $p_h$, 
    and $z_h$ for the 
    energy regularization approach with a discontinuous target, at $t=0.3$,
    $0.5$, and $0.7$ (from top to bottom).}
  \label{fig:ex4sol}
\end{figure} 

\begin{figure}
  \centering
  \includegraphics[width=0.24\textwidth]{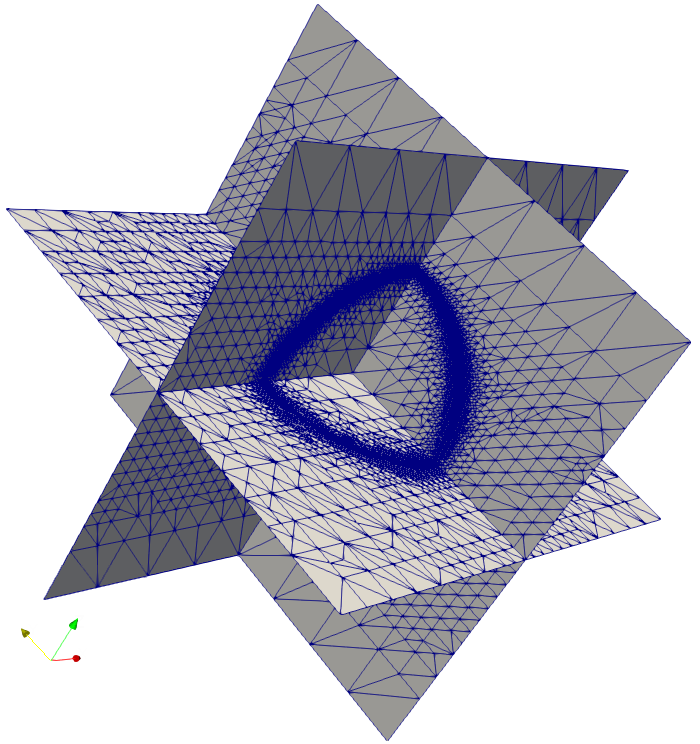}
  \includegraphics[width=0.24\textwidth]{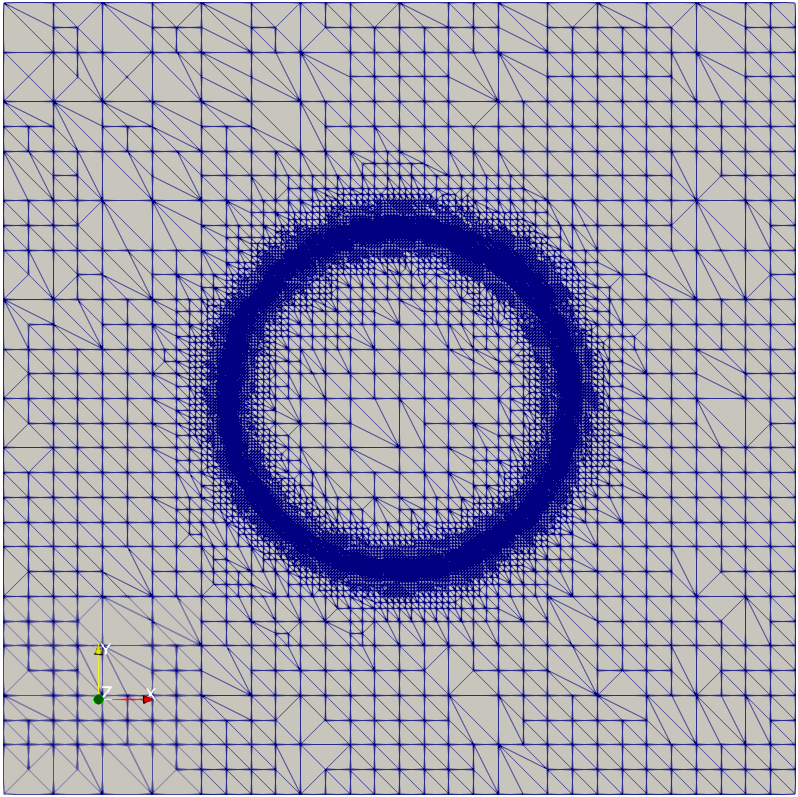}
  \includegraphics[width=0.24\textwidth]{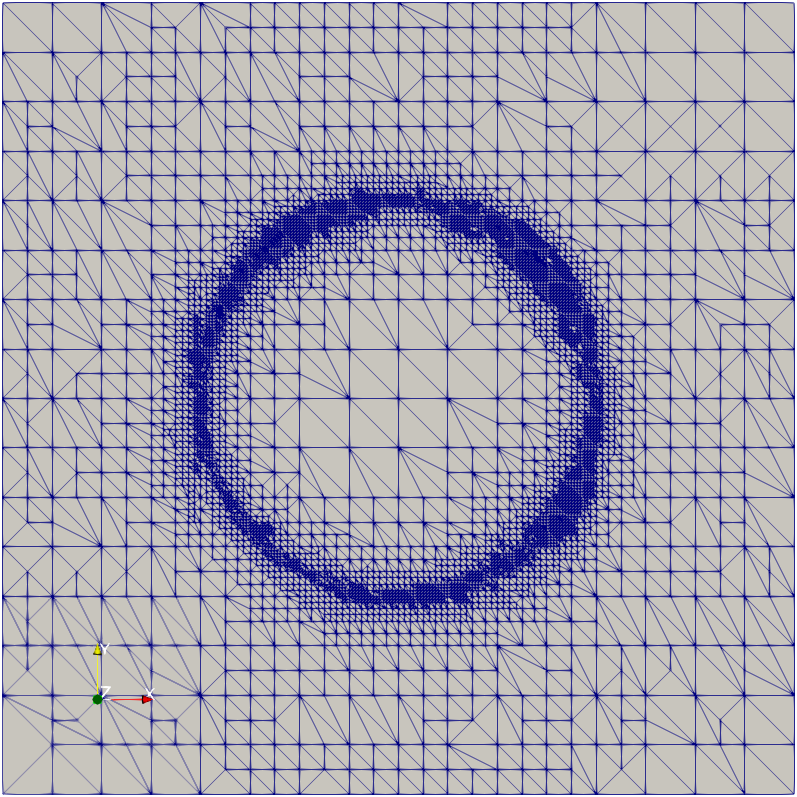}
  \includegraphics[width=0.24\textwidth]{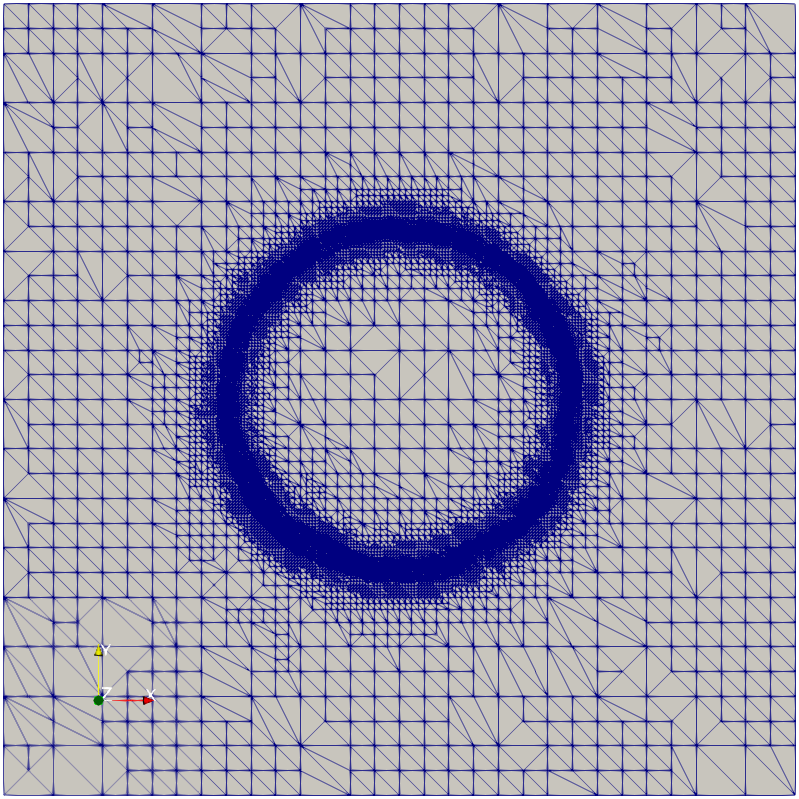}
  \caption{Example 2, Adaptive mesh refinement at the $78$th step in 
           space-time at $t=0.375$, $0.5$, and $0.625$ (from left to 
           right), using energy regularization with a discontinuous target.} 
  \label{fig:ex4mesh}
\end{figure} 

Moreover, we also compare the numerical solutions of the energy 
regularization approach (\ref{eqn:energy regularization}) with the 
solution using $L^2$-regularization (\ref{eqn:L2 regularization}), and 
with the solution using an $L^2+L^1$ regularization that promotes 
spatio-temporal sparsity, i.e., minimize
\begin{equation}\label{eq:L2L1 regularization}
{\mathcal J} (u, z) := 
\frac{1}{2} \, \int_Q \left| u  - u_d \right|^2\,dx\,dt +
\frac{1}{2}\varrho \, \|z\|_{L^2(Q)}^2 + \mu \, \|z\|_{L^1(Q)},
\end{equation}
subject to \eqref{eqn:linearstateequation}, and with $\mu > 0$. A similar 
parabolic optimal sparse control model problem has been
considered, e.g, in our recent work \cite{LSTY_RADON}, see also the
references given therein.
 
For all cases, we plot the state $u$ and the control $z$ at
$t=0.65$ having a closer look near the interface as depicted in
Fig. \ref{fig:compl2l2l1h1}. As predicted, the interface is resolved 
much sharper, and much less oscillations show up, for the state using the
energy regularization than using the $L^2$-regularization. With the
additional $L^1$ term, the state shows less oscillation than pure 
$L^2$-regularization, but a less shaper interface captured than the 
energy regularization. We further obtain sparser controls by the energy 
regularization, i.e., the control is non-zero only in a narrow region along 
the interface, while the control acts in a much larger region near the
interface and almost extends to the whole space-time domain by the 
$L^2$-regularization. The $L^2+L^1$ approach produces a bit spatially sparser 
solution than the $L^2$-regularization, see Fig. \ref{fig:compl2l2l1h1line} 
for a closer comparison of the control along the
line $[0,0.55,0.65]-[1,0.55,0.65]$. All these results are obtained on
adaptive meshes that are driven by residual-type error indicators for the
coupled optimality system. A comparison of adaptive meshes on the cutting plane
$t=0.625$ is illustrated in Fig. \ref{fig:compadaptmesh}.  

\begin{figure}
  \centering
  \includegraphics[width=0.32\textwidth]{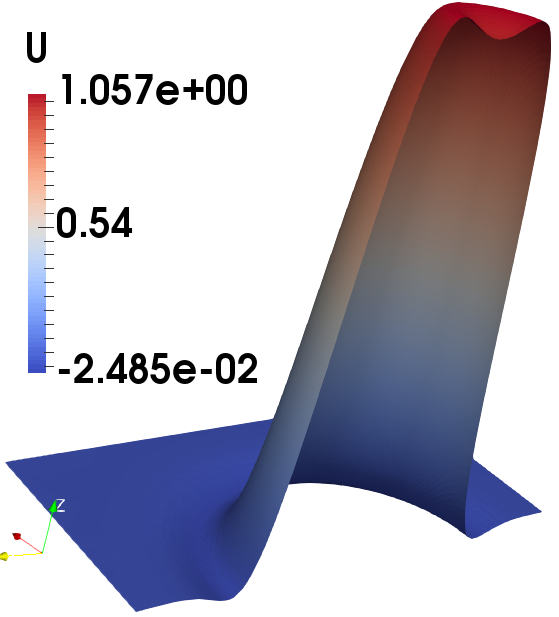}
  \includegraphics[width=0.32\textwidth]{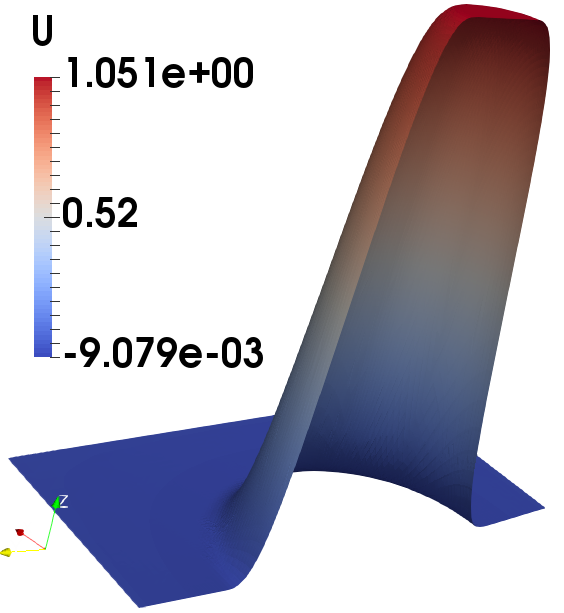}
  \includegraphics[width=0.32\textwidth]{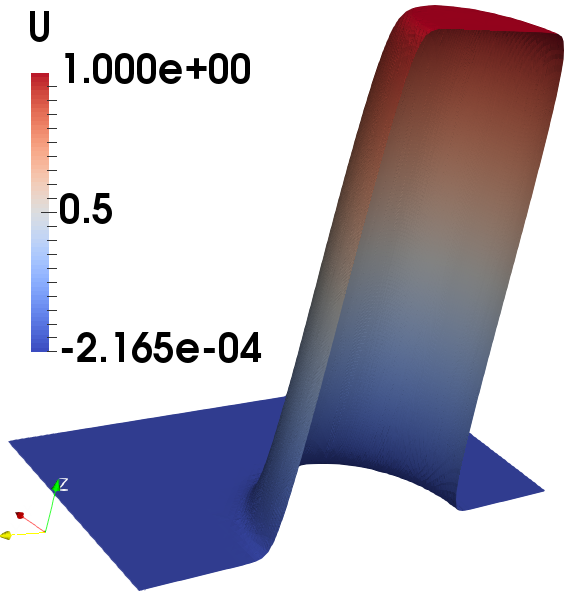}
  \includegraphics[width=0.32\textwidth]{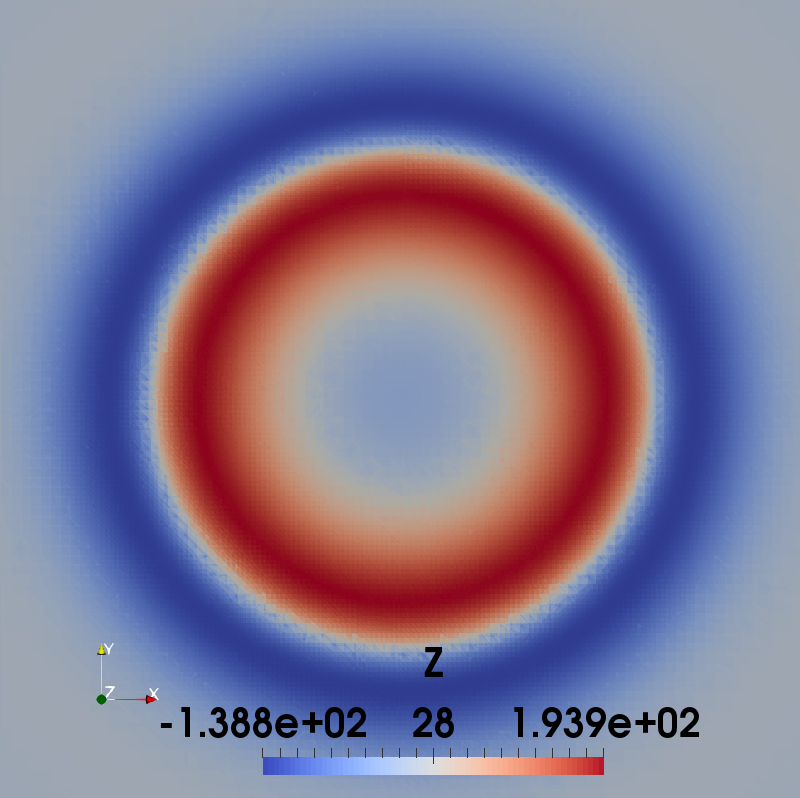}
  \includegraphics[width=0.32\textwidth]{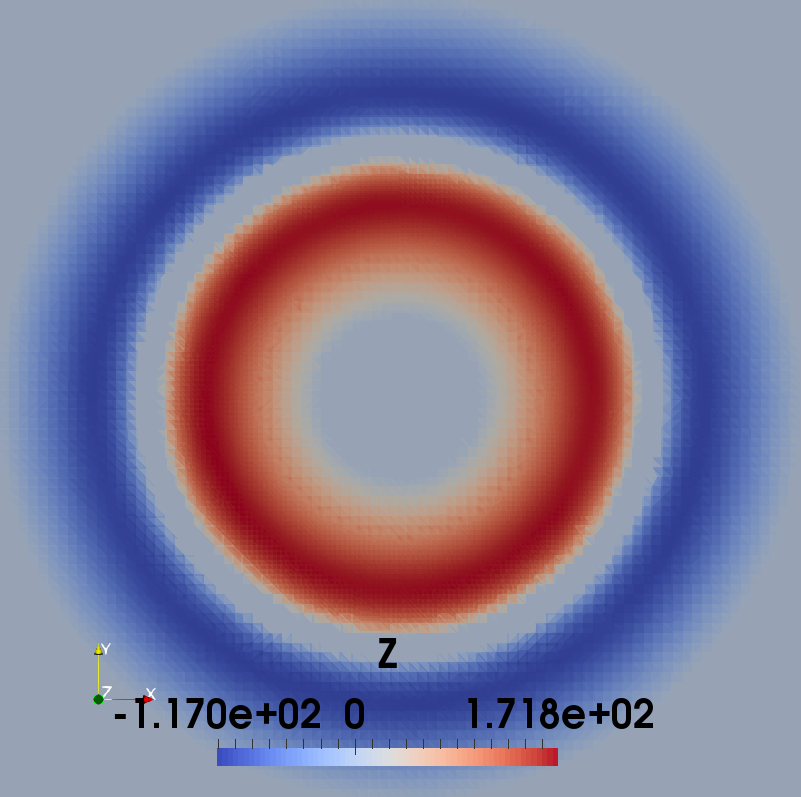}
  \includegraphics[width=0.32\textwidth]{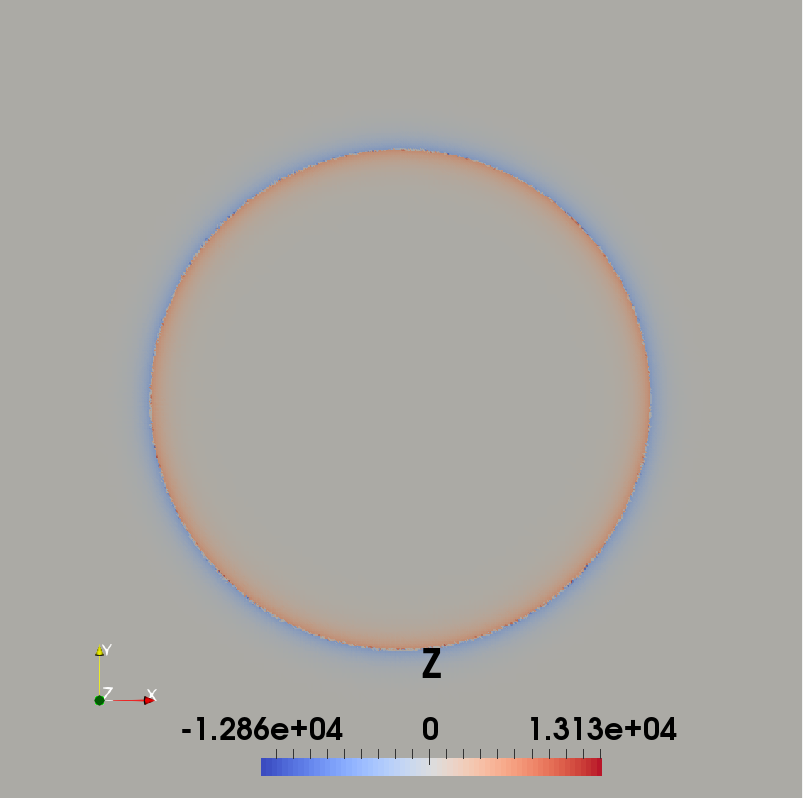}
  \caption{Example 2, Comparison of the numerical solutions of the state $u$
    and control $z$ (from top to bottom) at time $t=0.65$, using
    $L^2$-regularization ($\rho=10^{-6}$),
    $L^2+L^1$ ($\rho=10^{-6}$, $\mu=10^{-4}$), and energy 
    regularization (from left to right).} 
  \label{fig:compl2l2l1h1}
\end{figure}  

\begin{figure}
  \centering
  \includegraphics[width=0.32\textwidth]{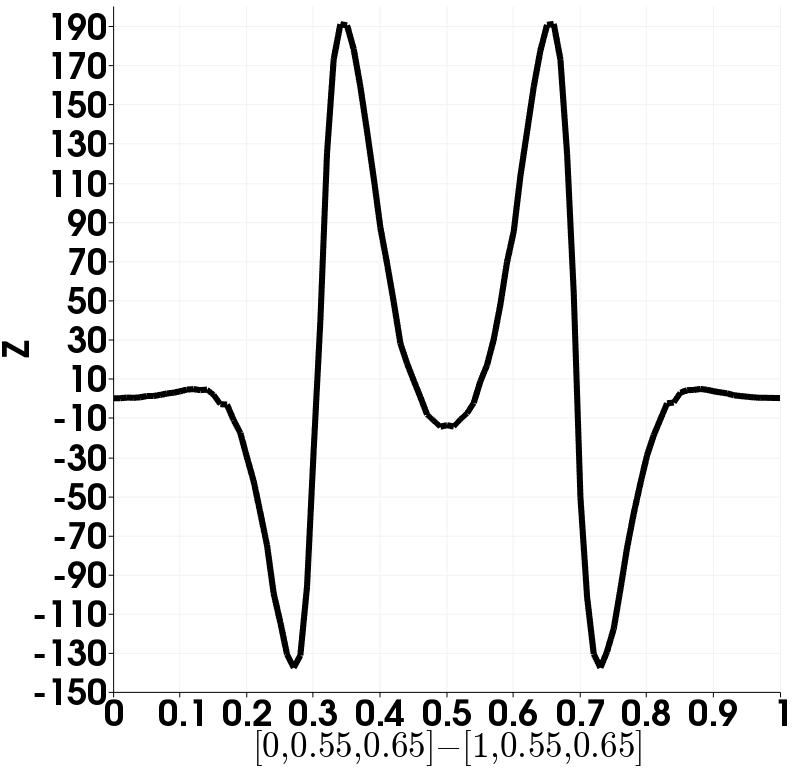}
  \includegraphics[width=0.32\textwidth]{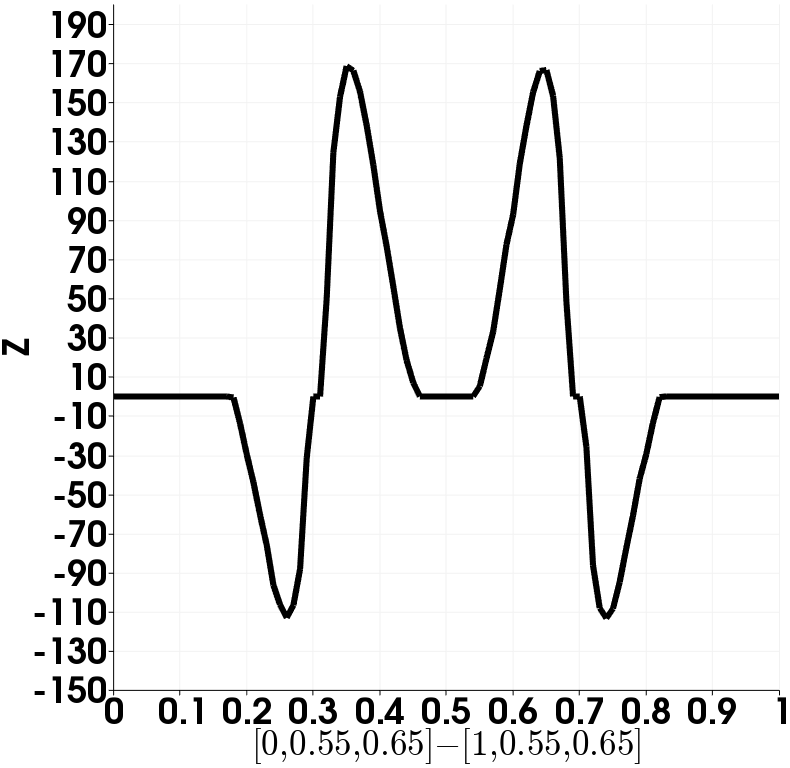}
  \includegraphics[width=0.32\textwidth]{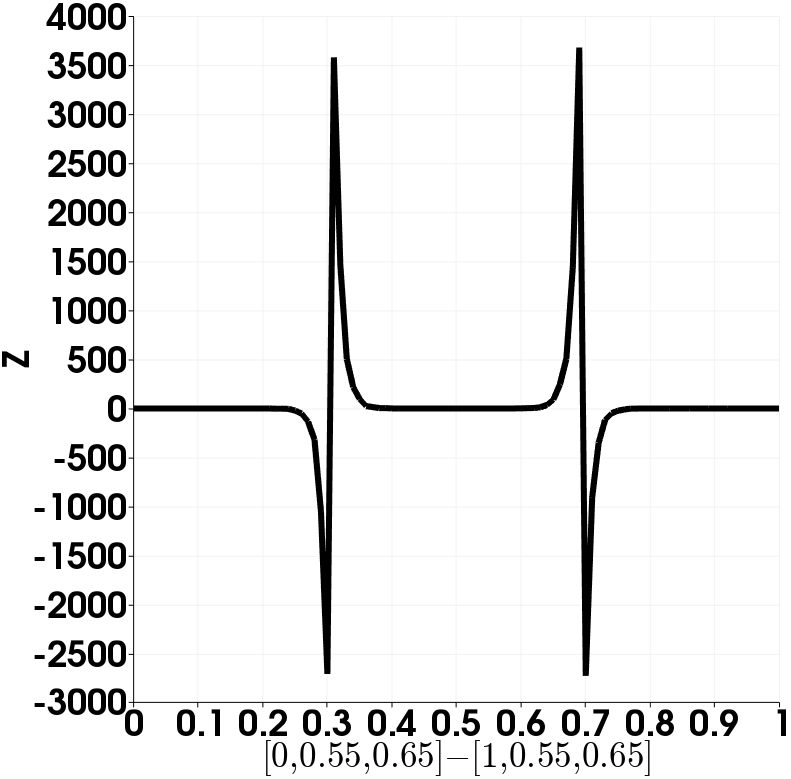}
  \caption{Example 2, Comparison of the numerical solutions of the control $z$ 
    along the line $[0,0.55,0.65]-[1,0.55,0.65]$, using $L^2$-regularization 
    ($\rho=1e-6$),
    $L^2+L^1$ ($\rho=1e-6$, $\mu=1e-4$), and energy 
    regularization (from left to right).} 
  \label{fig:compl2l2l1h1line}
\end{figure}  

\begin{figure}
  \centering
  \includegraphics[width=0.32\textwidth]{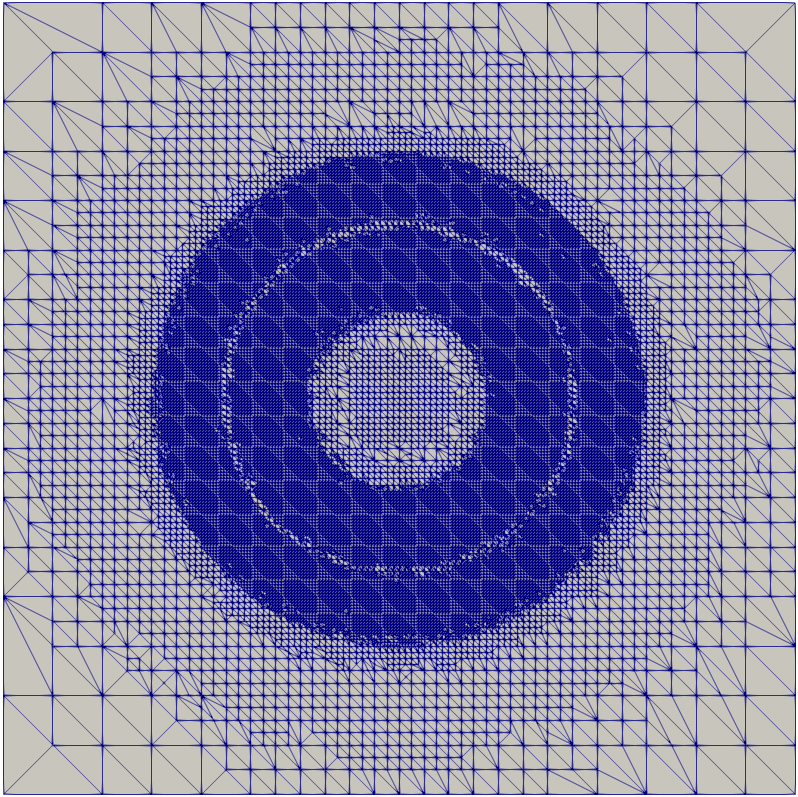}
  \includegraphics[width=0.32\textwidth]{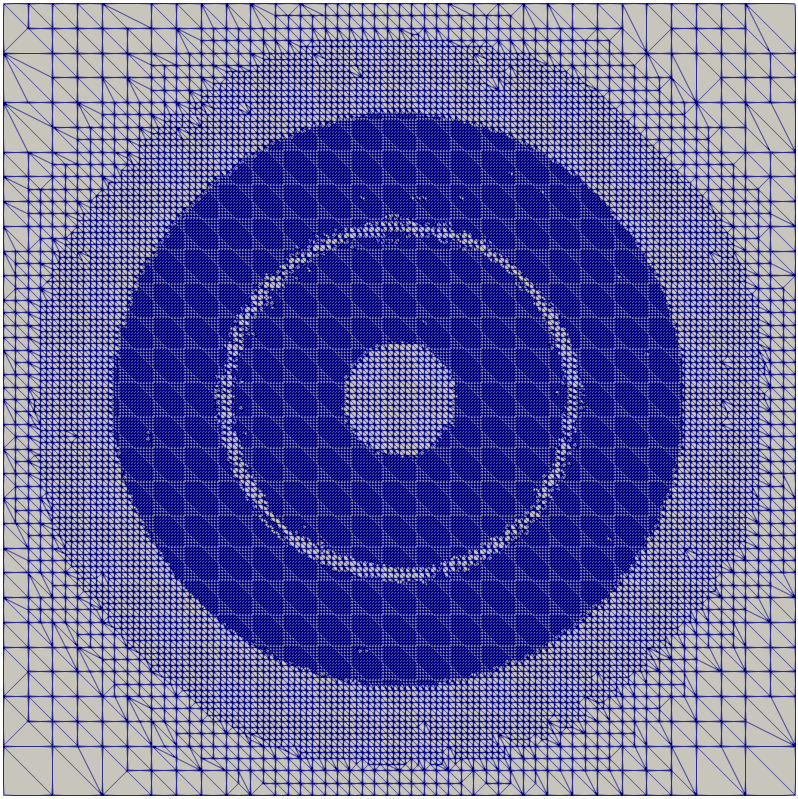}
  \includegraphics[width=0.32\textwidth]{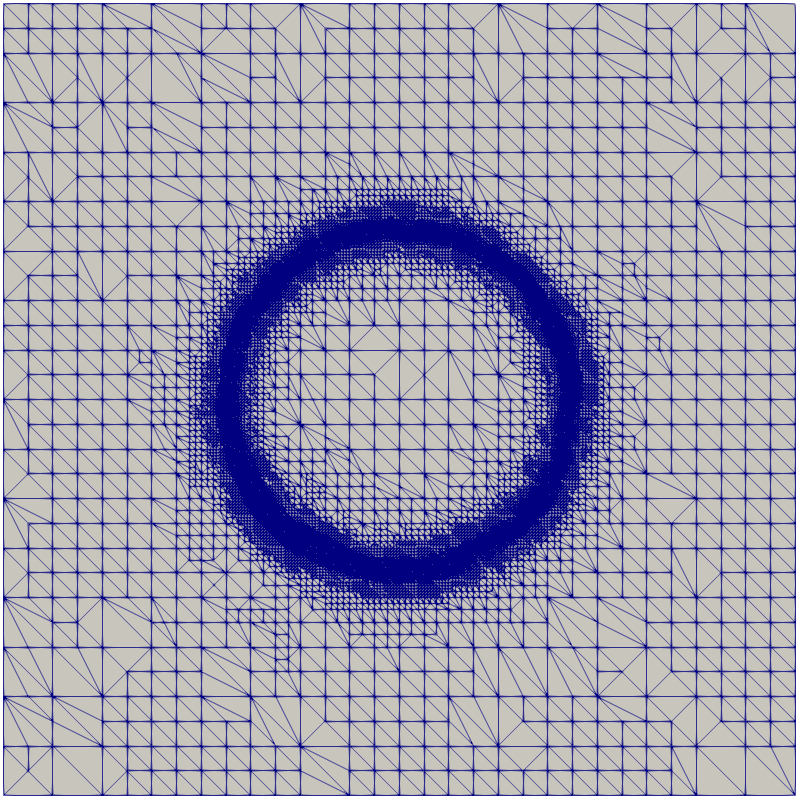}
  \caption{Example 2, Comparison of the adaptive meshes on the cutting plane at time
    $t=0.625$: 
    $L^2$-regularization, $20$th step, $2,080,493$ grid points in space-time (left);
    $L^2+L^1$, $23$th step, $3,320,340$ grid points in space-time (middle);
    energy regularization, $78$th step, $3,398,213$ grid points in space-time (right).} 
  \label{fig:compadaptmesh}
\end{figure}

\subsection{An example in three space dimensions}
Analogous to the example in two space dimensions as considered in 
Section \ref{sec:2dexm}, we now construct an example with an explicitly known solution of the
first-order optimality system in three space dimensions as follows:   
\begin{equation*}
  \begin{aligned}
    u(x,t)&=3\pi^2\sin(\pi x_1)\sin(\pi x_2)\sin(\pi x_3)\left(et^2+t\right),\\
    p(x,t)&=-\varrho\sin(\pi x_1)\sin(\pi x_2)\sin(\pi x_3)\left( at^2+bt+1\right),\\
    z(x,t)&=3\pi^2\sin(\pi x_1)\sin(\pi x_2)\sin(\pi x_3)\left( at^2+bt+1\right),
  \end{aligned}
\end{equation*}
where 
\[
a=-\frac{9\pi^4+3\pi^2}{3\pi^2+2}, \quad
b=\frac{9\pi^4-2}{3\pi^2+2}, \quad
e=-\frac{3\pi^2+1}{3\pi^2+2}. 
\]
When using the adjoint equation, the target is given as 
$u_d=u+\partial_tp+\Delta_xp$, and we set the
regularization parameter $\varrho=0.01$. The numerical solutions 
$u_h$ at $t=1$, $p_h$ at $t=0$ and $z_h$ at $t=0$ are
displayed in Fig. \ref{fig:4dsol}. 

\begin{figure}
  \centering
  \includegraphics[width=0.325\textwidth]{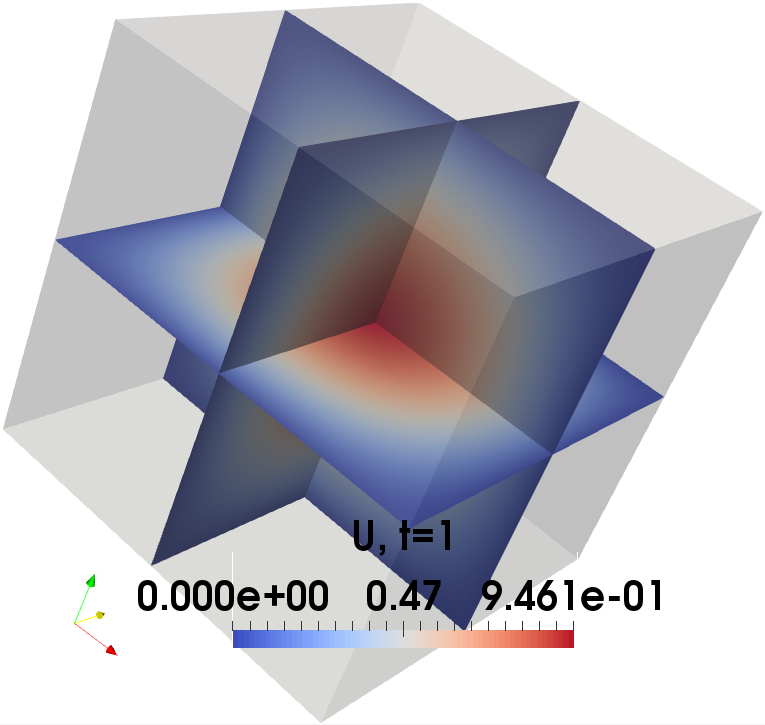}
  \includegraphics[width=0.325\textwidth]{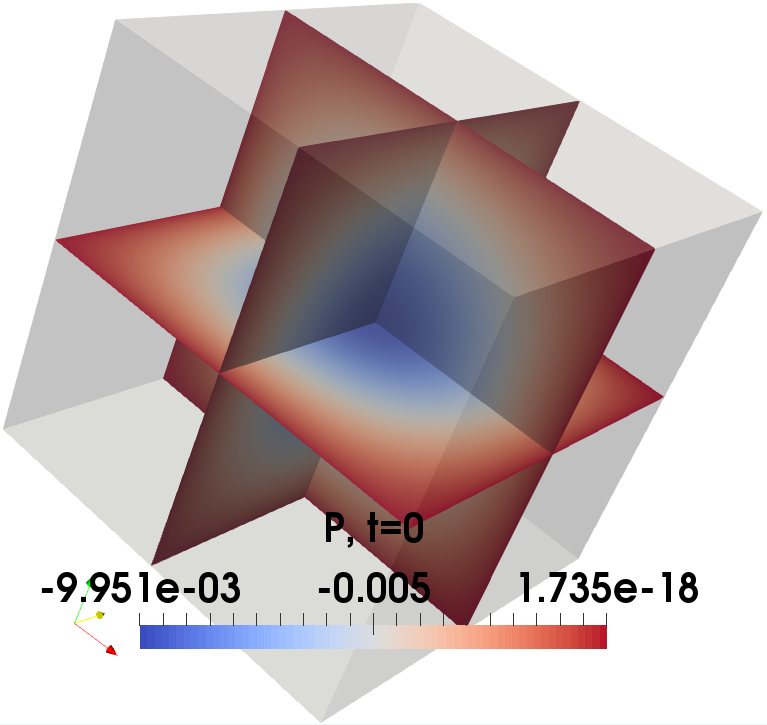}
  \includegraphics[width=0.325\textwidth]{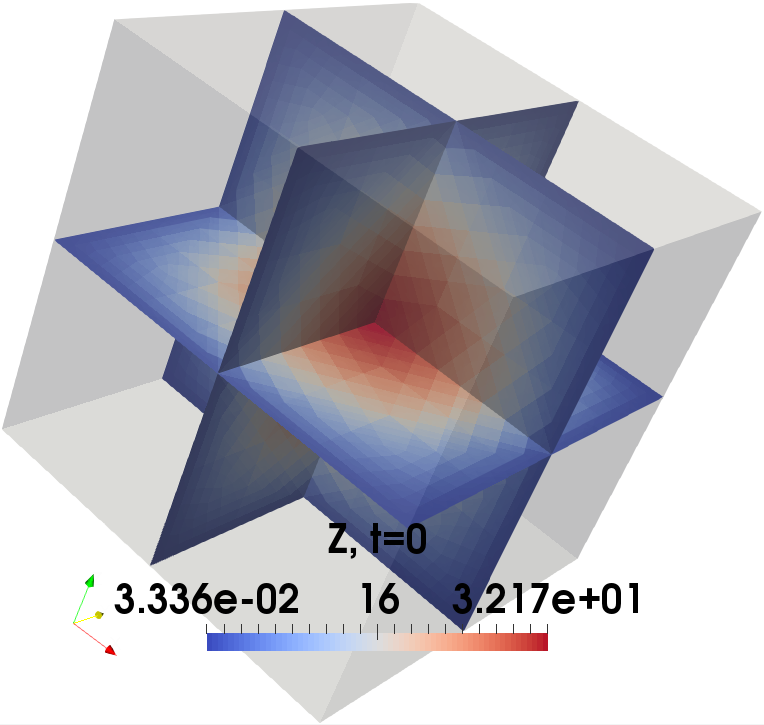}
  \caption{Example 3, numerical solutions $u_h$ at $t=1$, $p_h$ at $t=0$, and $z_h$ at
    $t=0$ (from left to right).} 
  \label{fig:4dsol}
\end{figure}  

The error of the space-time finite element approximations $u_h$ and $p_h$ in
the corresponding norms $\|\cdot\|_Y$ and $\|\cdot\|_{L^2(Q)}$ are given in 
Table \ref{tab:eoc4d} at every second refinement step, see also the error of
the objective functional $|J(u, z) - J(u_h, z_h)|$ in the last column. In addition,
we illustrate the convergence rates in Fig.~\ref{fig:4dconv}, where we observe optimal
rates for the state $u$ and the adjoint state $p$ as already experienced in two space
dimensions. Note that, at the finest refinement level, we have
$11,171,330$ degrees of freedom in total for the coupled system. The mesh size
is approximately $h=0.03125$.  

\begin{table}
\caption{Example 3, the error for numerical approximations $u_h$, $p_h$, and
  $J_h(u_h,z_h)$, with $J(u, z)=7.818e-1$.} \label{tab:eoc4d}
\centering
\begin{tabular}{rrrrrrr}
\toprule
Lev &\#Dofs &  $\|u-u_h\|_Y$ & $\|u-u_h\|_{L^2(Q)}$ & $\|p-p_h\|_Y$ & 
$\|p-p_h\|_{L^2(Q)}$ & $|J-J_h|$ \\
\midrule
$2$ &$470$ & $5.638e-0$ & $3.757e-1$ & $5.663e-2$& $3.911e-3$& $2.144e-1$    \\
$4$ &$1,430$ & $4.988e-0$ & $2.927e-1$ &$5.012e-2$& $3.096e-3$& $1.648e-1$\\
$6$ &$4,370$ &  $4.111e-0$ & $2.142e-1$ & $4.132e-2$& $2.251e-3$&$1.064e-1$ \\
$8$ &$18,450$ & $3.225e-0$ & $1.252e-1$ & $3.239e-2$& $1.321e-3$&$5.916e-2$  \\
$10$ &$53,186$ & $2.187e-0$ & $6.349e-2$ &  $2.195e-2$& $6.648e-4$&$2.559e-2$\\
$12$ &$268,226$ & $1.728e-0$ & $4.462e-2$ & $1.735e-2$& $4.617e-4$&$1.573e-2$  \\
$14$ &$744,962$ & $1.097e-0$ & $1.547e-2$ & $1.101e-2$& $1.633e-4$&$6.042e-3$  \\
$16$ &$4,103,682$  & $8.627e-1$ &$1.099e-2$ & $8.661e-3$& $1.141e-4$&$3.729e-3$ \\
$18$ &$11,171,330$  & $5.484e-1$ & $3.846e-3$ & $5.506e-3$&$4.074e-5$&$1.487e-3$ \\
\bottomrule
\end{tabular}
\end{table}

\begin{figure}
  \centering
  \includegraphics[width=0.326\textwidth]{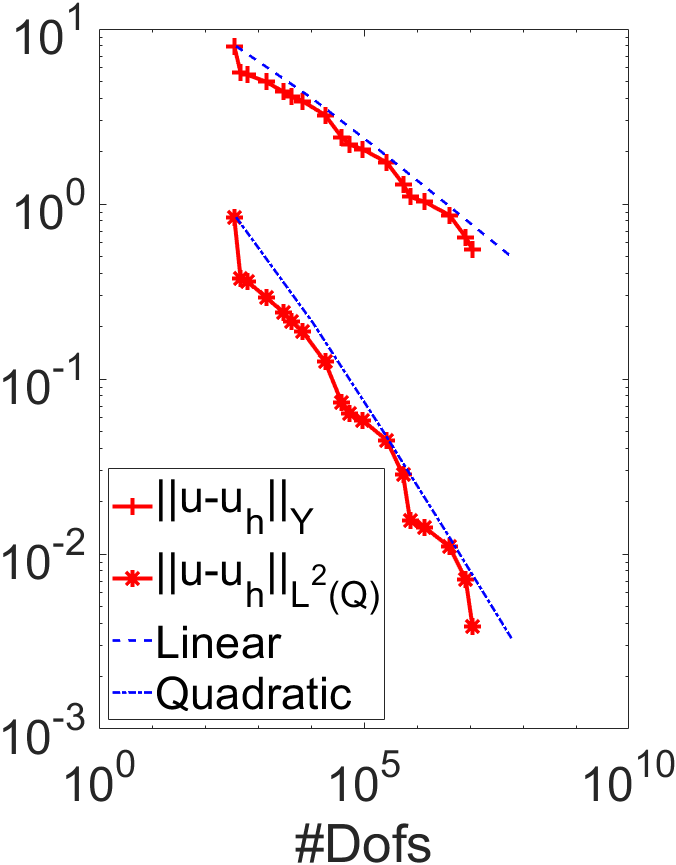}
  \includegraphics[width=0.326\textwidth]{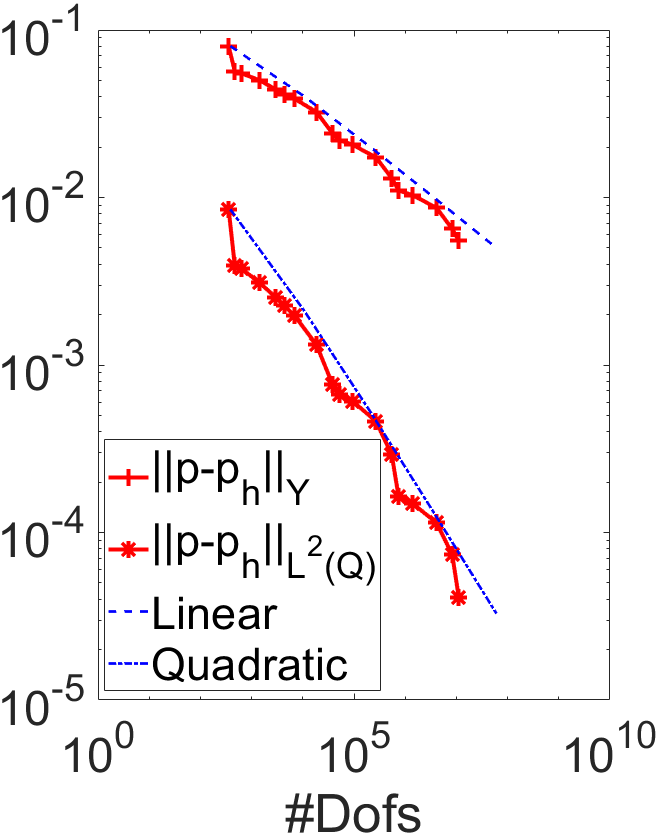}
  \includegraphics[width=0.326\textwidth]{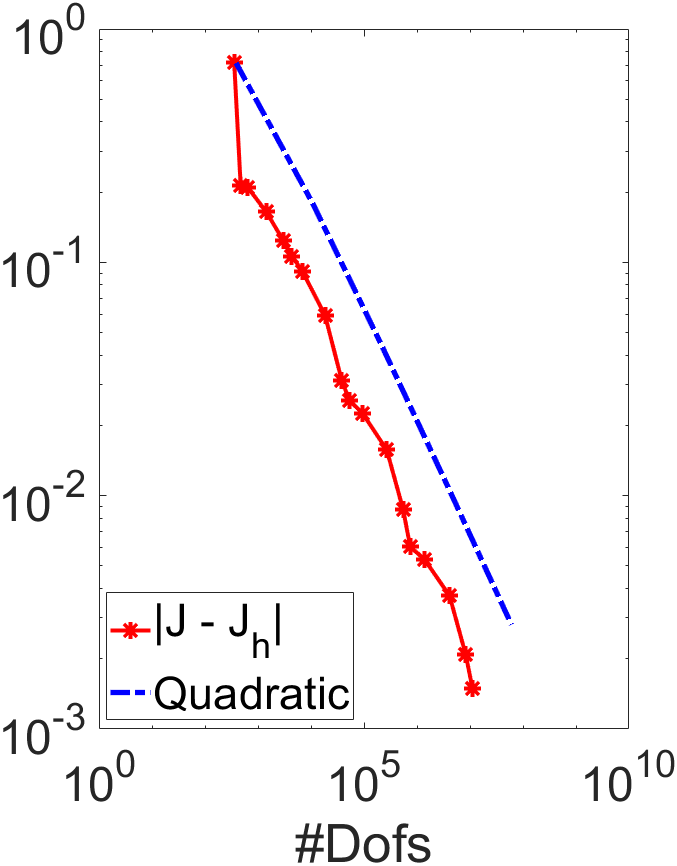}
  \caption{Example 3, convergence history for numerical approximations $u_h$, 
    $p_h$, and $J_h(u_h,z_h)$ (from left to right).} 
  \label{fig:4dconv}
\end{figure} 

\section{Conclusions}\label{sec:con}
In this work, we have analyzed space-time finite element methods for the
numerical solution of parabolic optimal control problems using energy
regulaization in the tracking type objective functional. In contrast to the
$L^2$ regularization approach as well as the combined $L^2+L^1$ approach, 
we observe a more localized control and sharper contours for discontinuous 
target functions when using energy regularization. In the latter,
the discrete optimality system is block skew-symmetric but positive definite,
which will allow us to construct optimal preconditioned iterative solution
strategies,see, e.g., 
\cite{LSTY:SchielaUlbrich:2014a,LSTY:SchulzWittum:2008a,LSTY:Zulehner:2011a}.
Although we only consider the case of unconstrained optimal control problems, 
this approach can be extended to problems with control constraints and will 
be considered elsewhere.

\bibliography{LSTY}
\bibliographystyle{abbrv}

\end{document}